\documentclass{article}
\usepackage[utf8]{inputenc} 
\usepackage[T1]{fontenc}    
\usepackage[english]{babel}

\usepackage{amsfonts}
\usepackage{nicefrac}       
\usepackage{microtype}      
\usepackage{lmodern}
\usepackage{amssymb,amsmath,amsthm}

\usepackage{stmaryrd}
\SetSymbolFont{stmry}{bold}{U}{stmry}{m}{n}

\usepackage{bbm,bm}
\usepackage{latexsym}
\usepackage{xcolor}

\usepackage{enumerate}
\usepackage{verbatim}
\usepackage{booktabs}       

\usepackage{url}            
\usepackage[colorlinks=true]{hyperref}
\usepackage[numbers,sort&compress,square,comma]{natbib}
\usepackage[capitalise]{cleveref}

\usepackage{geometry}
\usepackage{parskip}
\makeatletter
\def\thm@space@setup{
  \thm@preskip=\parskip
  \thm@postskip=0pt
}
\makeatother

\usepackage[short]{optidef}
\usepackage{algorithmic,algorithm}

\usepackage{thmtools}

\declaretheorem{theorem}
\declaretheorem{corollary}
\declaretheorem{lemma}

\declaretheorem{proposition}

\declaretheoremstyle[qed=$\square$]{definitionwithend}
\declaretheorem[style=definitionwithend]{definition}

\declaretheorem[style=definitionwithend]{example}
\declaretheorem[style=definitionwithend]{remark}

\crefname{assumption}{Assumption}{Assumptions}
\crefname{conjecture}{Conjecture}{Conjectures}
\crefname{fact}{Fact}{Facts}

\newcommand{\abs}[1]{\ensuremath{\left\lvert #1 \right\rvert}}

\newcommand{\by}{\times}
 
\newcommand{\norm}[1]{\ensuremath{\left\lVert #1 \right\rVert}}
\newcommand{\ip}[1]{\ensuremath{\left\langle #1 \right\rangle}}
\newcommand{\grad}{\ensuremath{\nabla}}

\newcommand{\set}[1]{\left\{#1\right\}}

\newcommand{\bs}{\boldsymbol}

\newcommand{\bbm}{\mathbbm}

\def\C{{\mathbb{C}}}

\def\R{{\mathbb{R}}}
\def\S{{\mathbb{S}}}

\def\cI{{\cal I}}

\def\cL{{\cal L}}

\def\cP{{\cal P}}
\def\cQ{{\cal Q}}
\def\cR{{\cal R}}

\DeclareMathOperator{\Diag}{Diag}
\DeclareMathOperator{\diag}{diag}
\DeclareMathOperator{\tr}{tr}

\DeclareMathOperator{\range}{range}

\DeclareMathOperator{\spann}{span}

\newcommand{\framedheader}[3]{
  \framebox[\textwidth]{
    \vbox{
      \vspace{2mm}
      \hbox to \textwidth {\hspace{1em}\today \hfill #1\hspace{1em}}
      \vspace{4mm}
      \hbox to \textwidth {\hfill \Large{#2} \hfill}
      \vspace{2mm}
    }
  }
  \vspace*{4mm}
}
\usepackage{soul}

\input{alex-dots}

\let\bs\undefined
\let\by\undefined
\let\bf\undefined

\newcommand{\bf}{\mathbf{f}}
\newcommand{\bg}{\mathbf{g}}
\newcommand{\bh}{\mathbf{h}}
\newcommand{\bs}{\mathbf{s}}
\newcommand{\bo}{\bbm{1}}

\newcommand{\bx}{\mathbf{x}}
\newcommand{\by}{\mathbf{y}}

\DeclareMathOperator{\sym}{sym}

\begin{document}

\title{A Unified Theory of H-Duality in First-Order Methods}
\author{
    Kevin Shu\footnote{University of Waterloo, Combinatorics and Optimization, \texttt{k6shu@uwaterloo.ca}}
    \and
    Alex L.\ Wang\footnote{Purdue University, Daniels School of Business, \texttt{wang5984@purdue.edu}}
}
	\date{\today}
\maketitle

\begin{abstract}
    We provide two complementary explanations of H-duality in smooth strongly convex optimization and contractive fixed-point problems.
H-duality refers to the phenomenon where the worst-case performance of many fixed-step first-order methods (FSFOMs) 
for a given performance setup
are exactly equal to the worst-case performance of their ``time-reversed'' counterparts and a paired performance setup.
This concept gained interest
as a mechanism for automatically converting an algorithm
with guarantees on the final suboptimality to algorithms with guarantees on the final gradient norm.
Since its initial discovery, additional H-dual performance setup pairings have been discovered empirically, 
yet a general explanation for these phenomena remained elusive.

Our first explanation of H-duality shows that, on a class of instances with extremal curvature properties, FSFOMs produce a final iterate that is a scalar multiple of the initial iterate, and that this scalar is invariant under the time reversal of the underlying FSFOM.
This, together with the empirical observation that many FSFOMs have such an instance with extremal curvature as their worst case, gives one explanation for why H-duality often occurs.
The second explanation is complementary and uses the notion of a performance estimation certificate of the convergence rate of a FSFOM. 
We exhibit an explicit transformation of such certificates that preserves many of the properties that define such a certificate, and we show that in many numerical and analytical examples, this transformation in fact transforms certificates of the convergence rate of an algorithm into certificates for its time-reversal.
As an application, we prove a convergence guarantee for the H-dual of ITEM matching the conjectured optimal gradient-to-gradient guarantee.

\end{abstract}

\section{Introduction}

This paper offers two explanations of H-duality~\cite{Kim2024-Hduality,kim2023mirror}, a recurring empirically observed phenomenon~\cite{kim2016optimized,grimmer2024composing,yoon2026theory,kim2024proof,Kim2021gradient} in first-order methods for smooth (strongly) convex optimization and contractive fixed-point problems.
For simplicity, this introduction focuses on smooth convex optimization while our results cover both settings.

H-duality was first observed as a numerical phenomenon in~\cite{Kim2021gradient,kim2016optimized}.
\citet{kim2016optimized} built on the the performance estimation problem (PEP) framework introduced in \cite{drori2012PerformanceOF} and gave a proof of the convergence rate of the Optimized Gradient Method (OGM). OGM is a first-order method providing a guarantee on the \emph{final suboptimality gap relative to the initial distance to the optimizer}. That is, OGM guarantees
\begin{align}
    \label{eq:DS_guarantee}
    h(y_N) - h(y_\star) \leq \tau_N^\textup{OGM}\cdot \frac{L}{2}\norm{y_0-y_\star}^2,
\end{align}
where $h$ is any $L$-smooth convex function $h:\R^d\to\R$ with minimizer $y_\star$, $y_0\in\R^d$ is any initial iterate, and $y_N$ is the output of OGM on the function $h$. Furthermore, assuming that $d\geq N+2$, \citet{drori2017exact} showed that OGM is optimal: no first-order method can guarantee 
\eqref{eq:DS_guarantee} with a strictly smaller constant $\tau_N$.

In follow-up work, \citet{Kim2021gradient} designed a first-order method OGM-G for minimizing the \emph{final gradient norm relative to the initial suboptimality}, guaranteeing
\begin{align}
    \label{eq:SG_guarantee}
    \frac{1}{2L}\norm{\grad h(y_N)}^2 \leq \tau_{N}^{\textup{OGM-G}} \cdot \left(h(y_0)- h(y_\star)\right).
\end{align}

While \eqref{eq:DS_guarantee} and \eqref{eq:SG_guarantee} both capture the idea that $y_N$ is ``near optimal'', it is not clear that there should be a direct relationship between these two guarantees. Surprisingly, \citet{Kim2021gradient} observed that the optimized rates were identical $\tau_N^\textup{OGM} = \tau_N^\textup{OGM-G}$. Furthermore, the algorithm OGM-G is \emph{exactly} the antitranspose (intuitively a ``time-reversed'' version) of OGM.

Beyond the example above, H-duality is more generally a phenomenon concerning the convergence guarantees of fixed-step first-order methods (FSFOMs).
A FSFOM is any algorithm that updates the query point by a predetermined linear combination of the previously observed gradients and can be represented by (or identified with) a strictly lower triangular matrix $\tilde H$.
Many familiar algorithms including Nesterov's fast gradient method~\cite{Nesterov1983} and gradient descent fall in this class.

H-duality is the following observation

\fbox{\parbox{\textwidth}{\centering Worst-case performance is often, but not always, preserved when replacing a method $\tilde H$ 
and a performance setup (initial and terminal performance criteria) with its dual method $\tilde H^A$ and its dual performance setup.}}

Here, the dual method is the \emph{antitranspose} $\tilde H^A$, obtained by flipping $\tilde H$ across its antidiagonal (see \cref{sec:preliminaries} for formal definitions).
In the OGM and OGM-G example, the performance setups were
initial distance to terminal suboptimality and initial suboptimality to terminal gradient respectively.

In general, there are many ways of measuring the 
``quality'' of a given point. For example, in the smooth convex setting, natural crtiera are the suboptimality, distance to the minimizer, and gradient norm:
\begin{gather}
    h(y) - h(y_\star) \tag{S}\\
    \|y - y_\star\|^2 \tag{D}\\
    \|\nabla h(y)\|^2 \tag{G}.
\end{gather}
A performance setup then consists of a pair of these crtieria. The 9 possible performance setups in $\set{\textup{S},\textup{D},\textup{G}}^2$ have been arranged in the following conjectured H-dual pairs based on numerical and theoretical evidence~\cite{taylor2022optimal,gupta2023branch,kim2024proof}:
\begin{align}
    \label{eq:function_dualities}
    &\textup{(D,S)}\leftrightarrow\textup{(S,G)}
    &&\textup{(G,S)}\leftrightarrow\textup{(S,D)}
    &&\textup{(S,S)}\leftrightarrow\textup{(S,S)}\\
    \label{eq:classical_dualities}
    &\textup{(D,G)}\leftrightarrow\textup{(D,G)}
    &&\textup{(D,D)}\leftrightarrow\textup{(G,G)}
    &&\textup{(G,D)}\leftrightarrow\textup{(G,D)}.
\end{align}

\begin{remark}
\label{rem:}
The dualities in the first row are defined only in the function minimization setting, whereas the dualities in the second row are well-defined in both the function minimization and fixed-point settings (the ``gradient'' in the operator setting should be interpreted as the fixed point residual; see \cref{sec:preliminaries}). For this reason, the main body of this paper will focus on the dualities in the second row.
Due to the historical significance of the (D,S)-(S,G) pairing, we translate our results from the main body to that setting in \cref{app:suboptimality}. 
We expect that analogues of our results hold for the remaining two pairings.
\end{remark}

\subsection{Our contributions}



While there is a relatively simple algebraic relationship mapping a FSFOM $\tilde H$ to its antitransposed FSFOM $\tilde H^A$, the mechanism by which the guarantees are preserved has remained incompletely understood.
Indeed, most analyses of H-dual pairs treat the two algorithms entirely separately~\cite{kim2016optimized,Kim2021gradient,grimmer2024composing}.
In this paper, we offer two explanations for when this phenomenon holds (and when it does not) and the underlying mechanisms.

Our first explanation (\cref{sec:extremal}) is based on the worst-case instances for a given method.
The key observation is that for many algorithms and performance setups, the worst-case behavior is attained on a function or operator exhibiting extremal curvature at the queried iterates.
We show that FSFOMs on these extremal trajectories are governed by a single scalar parameter that is invariant under antitransposition.
Whenever the worst-case instances for these algorithms have extremal curvature in this sense and the performance setups satisfy a simple compatibility condition, this invariance establishes H-duality.
This argument applies simultaneously to both the smooth strongly convex function setting and the contractive operator setting.
Moreover, this argument \emph{predicts H-dualities beyond the dualities in \eqref{eq:function_dualities} and \eqref{eq:classical_dualities}}.

Our second explanation (\cref{sec:congruences}) is complementary and works directly with PEP certificates.
Informally, a PEP certificate proves a convergence rate for $\tilde H$ by aggregating interpolation inequalities governing the problem class with nonnegative multipliers and verifying that the resulting quadratic form is PSD. These multipliers can be arranged into a matrix $\Lambda$ satisfying certain sign conditions.
For each of the dualities in \eqref{eq:classical_dualities}, we show that there is a common explicit transformation that maps any given matrix $\Lambda$ to another matrix $\Phi$ for which the quadratic forms associated with $(\tilde H, \Lambda)$ and $(\tilde H^A,\Phi)$ are congruent.
This algebraic correspondence does not rely on any sparsity assumption on $\Lambda$ or $\Phi$ and reduces the question of H-duality to whether $\Lambda$ and $\Phi$ satisfy the necessary sign conditions. 
In numerical experiments, among the (D,G)-(D,G) H-dual pairs $\tilde H$ and $\tilde H^A$ that we tested, nearly every pair had PEP certificates $\Lambda$ and $\Phi$ satisfying this algebraic relationship (see \cref{rem:beyond_tridiagonal}).

\cref{sec:tridiagonal} identifies the setting in which $\Lambda$ or $\Phi$ is tridiagonal as an important and easily verifiable setting in which the map $\Lambda\leftrightarrow\Phi$ \emph{automatically} preserves the sign conditions. We study this setting in detail because it makes the sign preservation automatic and includes certificates from many optimized/optimal FSFOMs. 
As an application, we use this transformation to prove a (G,G) convergence guarantee for the H-dual of ITEM~\cite{drori2017exact,Drori2021OnTO}, which had previously been conjectured and only confirmed numerically for small iteration counts.

\subsection{Prior work}



Following the discovery of OGM and OGM-G, \citet{Kim2024-Hduality} coined the term H-duality to describe this phenomenon and provided the first explanation of H-duality for the initial distance to terminal suboptimality (D,S) and initial suboptimality to terminal gradient (S,G) pairings.
Interpreted in our language, \citet{Kim2024-Hduality} provide an explicit invertible transformation between bidiagonal (D,S) PEP certificates $\Lambda$ with prescribed optimizer multipliers and bidiagonal-plus-last-row (S,G) PEP certificates $\Phi$.
They provide a two-way correspondence $\Lambda\leftrightarrow \Phi$ between PEP certificates with the prescribed supports.

More recently,~\citet{kim2024proof} extended these results to the smooth strongly convex setting as part of proving the exact convergence rate of constant step gradient descent.
\citet{kim2024proof} describes a general procedure for searching for a (D,S) PEP certificate $\Lambda$ 
given a (S,G) PEP certificate $\Phi$ (and vice versa). Specifically,~\cite{kim2024proof} shows that the relevant quadratic forms are congruent if $\Lambda$ and $\Phi$ satisfy a number of bilinear constraints.
Thus, when $\Phi$ is fixed, a candidate $\Lambda$ can be found by solving a linear system (and vice versa). It is unclear, however, 
whether this linear system is invertible and what to do if it is not uniquely solvable.
This is clarified in the setting where $\Phi$ has tridiagonal-plus-last-row support where an explicit map $\Phi\mapsto \Lambda$ is constructed.


The prior results of~\cite{kim2024proof,Kim2024-Hduality} are designed for the (D,S)--(S,G) pairing and exploit multiplier structures special to that setting. Thus, the results therein do not apply to the dualities in \eqref{eq:classical_dualities} nor to the contractive operator setting. Moreover, the explicit transformations $\Lambda\mapsto \Phi$ and $\Phi\mapsto\Lambda$ are only provided assuming particular sparsity patterns.
To our knowledge, there is no prior explanation for any of the dualities in \eqref{eq:classical_dualities} nor in the contractive operator setting.

We extend this certificate-level viewpoint in three directions. First, we identify the natural subspaces for the PEP certificates in \eqref{eq:classical_dualities} and provide an explicit invertible map between general invertible (reduced) PEP certificates. We make no assumption on the sparsity pattern or the multipliers involving the optimizer. Second, a single mechanism applies to both the contractive operator and smooth strongly convex function setting and across all three dualities in \eqref{eq:classical_dualities}. Third, the transformation isolates a single remaining obstruction to H-duality: whether $\Phi$ satisfies the necessary sign conditions.

Concurrent work by~\citet{yoon2026theory} builds upon an initial draft of the current manuscript, which we shared privately, to fully characterize the minimax optimal algorithms in the nonexpansive operator setting. Specifically, they take \cref{thm:DG} as a starting point and study a combinatorial structure underlying when the map $\Lambda\leftrightarrow \Phi$ preserves the necessary sign conditions on optimal algorithms.
Their results complement our framework: the present paper identifies the common certificate transformation across settings and performance criteria, whereas their work analyzes the sign preservation question for particular FSFOMs in a fixed performance setup.



\section{Preliminaries: PEP, interpolation, and antitransposition}
\label{sec:preliminaries}

This section introduces the notation and definitions that we will use to simultaneously describe the smooth strongly convex function setting and the contractive operator setting (henceforth, the function and operator settings).

\subsection{Problem classes and fixed-step first-order methods}

\paragraph{The smooth strongly convex setting.} In this setting, our goal is to find an approximate minimizer to
    $\min_{y\in\R^d}h(y)$
where $h:\R^d\to\R$ is assumed to be $L$-smooth and $\mu$-strongly convex for some $\mu\in[0,L)$ and to possess a minimizer $y_\star$. Without loss of generality, we will assume $L=1$ and, by translating and shifting if necessary, that $y_\star = 0$ and $h(y_\star) = 0$.

The standard oracle model in this setting takes a query point $y_i$ and returns the gradient
\begin{align*}
    s_i = \grad h(y_i).
\end{align*}
We will also use the shorthand $h_i = h(y_i)$.

\paragraph{The contractive operator setting.}
In this setting, our goal is to find a fixed point of an operator $T:\R^d\to\R^d$ that is assumed to be $\frac{1-\mu}{1+\mu}$-contractive, where $\mu\in[0,1)$. That is
\begin{align}
    \label{eq:contraction}
    \norm{T(y) - T(z)} \leq \frac{1-\mu}{1+\mu}\norm{y-z}\qquad\text{for all $y,z\in\R^d$}.
\end{align}
We will assume that $T$ has a fixed point $y_\star$. By translating if necessary, we may assume without loss of generality that $y_\star = 0$.

The standard oracle model in this setting takes a query point $y_i$ and returns $T(y_i)$. We will associate to this query the scaled residual $s_i = \frac{1+\mu}{2}(y_i - T(y_i))$.

\begin{remark}
Our parameterization of the contraction parameter in \eqref{eq:contraction} and the scaled residual $s_i$ in the operator
setting aligns it with the function setting: Given a $1$-smooth $\mu$-strongly convex function, define the gradient
descent map $T_h(y) \coloneqq y -\frac{2}{1+\mu} \grad h(y)$. Then, $y_\star$ is a minimizer of $h$ if and only if it is a fixed point of $T_h(y)$. Furthermore,  the operator $T_h$ is $\frac{1-\mu}{1+\mu}$-contractive and
\begin{equation*}
    \grad h(y) = \frac{1+\mu}{2}(y - T_h(y)) \qquad\text{for all $y\in\R^d$}.\qedhere
\end{equation*}
\end{remark}

\paragraph{Fixed-step first-order methods.}
Now, let $\cI_N = \set{0,1,\dots,N}$.
In both settings, a fixed-step first-order method (FSFOM) is represented by a strictly lower triangular matrix $\tilde H \in\R^{\cI_N\times \cI_N}$. This matrix defines the following algorithm:
\begin{align*}
    &y_0\in\R^d\qquad\text{given}\\
    &y_n = y_{n-1} - \sum_{i=0}^{n-1} \tilde H_{n,i} s_i\in\R^d\qquad \text{for }n = 1,\dots,N.
\end{align*}
Recall that in the function setting $s_i$ is the gradient of the underlying function $h$, whereas $s_i$ is the scaled residual in the operator setting.

Throughout, we will interpret elements of $\R^d$ (such as $y_i$ and $s_i$) as row vectors.
Let $\by$ and $\bs$ denote matrices in $\R^{\cI_N\times d}\simeq (\R^d)^{\cI_N}$ with $i$th row $y_i$ and $s_i$.
For the index $\star$, define $y_\star = 0$ and $s_\star = 0$.

In the function setting, additionally define $\bh$ to be the vector with $i$th entry $\bh_i = h_i$ and set $h_\star = 0$.

We can rewrite the FSFOM as
\begin{equation*}
    \begin{bmatrix}
    y_0\\
    y_1 - y_0\\
    \vdots\\
    y_N - y_{N-1}
    \end{bmatrix} = - \tilde H \bs + e_0 y_0.
\end{equation*}
Let $S$ denote the cumulative sum operator, i.e., the square matrix with ones on and below the diagonal.
Left-multiplying both sides of the above equation by $S$ gives the compact expression
\begin{align*}
 \by = -S\tilde H\bs + \bbm{1} y_0.   
\end{align*}

\subsection{Antitransposition}

Let $R$ denote the reversal operator and let $P\coloneqq SR$. Pictorially,
\begin{align*}
    R = \begin{bmatrix}
    &&1\\
    &\bddots\\
    1
    \end{bmatrix},\qquad P = SR = \begin{bmatrix}
        &  &1\\
         & \bddots  & \vdots\\
        1  & \dots & 1
        \end{bmatrix}.
\end{align*}

\begin{definition}
Given a square matrix $\tilde H$, define its antidiagonal transpose (antitranspose) to be
\begin{align*}
    \tilde H^A  \coloneqq R \tilde H^\intercal R.
\end{align*}
If $X^A= X$, we say that $X$ is \emph{self-antitranspose}.
\end{definition}

\subsection{The transformed trajectory data}

The trajectory data $(\by,\bh,\bs)$ in the function setting and $(\by,\bs)$ in the operator setting are awkward to work with directly. We define transformed data that will simplify calculations later on.

In both settings, parameterize $\mu = \frac{q}{1+q}$, where $q\in[0,\infty)$. Equivalently, $q = \frac{\mu}{1-\mu}$. Then define
$x_i \coloneqq y_i - s_i$ and 
$g_i \coloneqq s_i - q x_i$.
Collecting these objects into their own elements of $\R^{\cI_N\times d}$, these definitions are equivalently
\begin{gather*}
    \bx \coloneqq \by - \bs,\qquad
    \bg \coloneqq \bs - q\bx.
\end{gather*}
The corresponding quantities for the index $\star$ are 
$x_\star = y_\star = 0$, $g_\star = s_\star = 0$.

In the function setting, additionally define
    $f_i \coloneqq h_i - \frac{1}{2}\norm{s_i}^2 - \frac{q}{2}\norm{x_i}^2$, or, in vector notation,
\begin{align*}
    \bf &\coloneqq \bh - \frac{1}{2}\diag(\bs\bs^\intercal) - \frac{q}{2}\diag({\bx\bx^\intercal}).
\end{align*}
The corresponding quantity for the index $\star$ is $f_\star = h_\star = 0$.

Below, we write the trajectory $\bx$ as a function of $\bg$ and the initial query point $y_0$.
Define
\begin{align}
    \label{eq:H_vH_definition}
    H &\coloneqq (S^{-1} + \tilde H)((1+q)I + q S\tilde H)^{-1}\qquad\text{and}\qquad
    v_H \coloneqq (I - q S H)\bbm{1}.
\end{align}

\begin{lemma}
    \label{lem:x_identity}
It holds that
    $\bx = -S H \bg + v_H y_0$.
\end{lemma}
\begin{proof}
In this proof only, set $B = (1+q)I + qS\tilde H$ so that $H = (S^{-1} + \tilde H)B^{-1}$.

We can write
\begin{align*}
    \bx &= \by - \bs=-\left(I + S\tilde H\right)\bs + \bbm{1} y_0= -\left(I +S\tilde H\right)(\bg +q\bx) + \bbm{1} y_0 = -(B-I)\bx -\left(I +S\tilde H\right)\bg + \bbm{1} y_0.
\end{align*}
Rearranging gives
\begin{align*}
    \bx &= -B^{-1}(I + S\tilde H)\bg + B^{-1}\bbm{1} y_0.
\end{align*}

The term involving $\bg$ is
\begin{align*}
    -B^{-1}(I + S\tilde H)\bg &= -(I + S\tilde H)B^{-1}\bg = - S H \bg,
\end{align*}
where the first equality follows as $I$ and $S\tilde H$ commute.

Next, note that
\begin{align*}
    I - qSH &= I - q(I + S\tilde H)B^{-1}
    = (B - q(I+ S\tilde H))B^{-1}
    = B^{-1}.
\end{align*}
Thus, the term involving $y_0$ is $B^{-1}\bbm{1} y_0= v_H y_0$.
\end{proof}

The H-duality phenomenon maps convergence guarantees between a FSFOM $\tilde H$ and its antitranspose $\tilde H^A$. The following lemma verifies that the operation of mapping $\tilde H$ to $H$ commutes with taking the antitranspose. Thus, we may equivalently study convergence guarantees that are preserved between $H$ and $H^A$.

\begin{lemma}
The operation $\tilde H\mapsto H$ commutes with the antitranspose operation.
\end{lemma}
\begin{proof}
Let $\hat H = \tilde H + S^{-1}$.
By definition
\begin{align*}
    H = (\tilde H +  S^{-1})((1+q)I + q  S\tilde H)^{-1} = \hat H (I + q  S \hat H)^{-1}.
\end{align*}
Thus, we can write the map $\tilde H\mapsto  H$ as a composition of the maps $\tilde H\mapsto \hat H$ and $\hat H\mapsto H$. Within this proof, denote these two maps by $F_1(\tilde H)$ and $F_2(\hat H)$. Our goal is to show that
\begin{align*}
    F_2(F_1(\tilde H^A)) = F_2(F_1(\tilde H)^A) = F_2(F_1(\tilde H))^A.
\end{align*}

For the first equality, we compute
\begin{align*}
    F_1(\tilde H)^A &= (\tilde  H +  S^{-1})^A= R(\tilde H^\intercal +  S^{-\intercal})R\\
    &= R\tilde H^\intercal R + R S^{-\intercal}R= \tilde H^A +  S^{-A}\\
    &= F_1(\tilde  H^A).
\end{align*}
Here, we have used that $ S^{-1} =  S^{-A}$.

Now, let $\hat H = F_1(\tilde H)$. 
By definition,
\begin{align*}
     H = \hat H(I + q S\hat H)^{-1} = (I + q\hat H S)^{-1} \hat H.
\end{align*}
Here, the second identity can be verified by left-multiplying both sides by $(I+q\hat H S)$ and right-multiplying both sides by $(I + q S\hat H)$:
\begin{align*}
    \hat H + q\hat H S\hat  H = \hat H + q\hat H S\hat H.
\end{align*}
Then,
\begin{align*}
    F_2(\hat H)^A &= ((I + q\hat H S)^{-1} \hat H)^A= R \hat H^\intercal (I + q  S^\intercal\hat H^\intercal)^{-1}R\\
    &= \hat H^A(I + q  S\hat H^A)^{-1}=F_2(\hat H^A).\qedhere
\end{align*}
\end{proof}


\subsection{Interpolation}
The following results are the foundation of the performance estimation problem (PEP) literature~\cite{taylor2017interpolation,Ryu2020tightSplittingContraction}.
They provide necessary and sufficient conditions for a finite set of data to be interpolated by a smooth strongly convex function or a contractive operator.
The transformations on the trajectory data defined earlier in this section can be used to greatly simplify the expressions for these interpolation inequalities.


\begin{lemma}
    \label{lem:interpolation}
\begin{itemize}
    \item In the function setting, the data $(y_i, h_i, s_i)_{i\in \cI_N\cup\set{\star}}$ can be interpolated by a $1$-smooth $\mu$-strongly convex function if and only if
\begin{align*}
f_j - f_i - \ip{g_i, x_j - x_i}\geq 0\qquad\text{for all }i,j\in \cI_N\cup\set{\star}.
\end{align*}
\item In the operator setting, the data $(y_i, s_i)_{i\in \cI_N\cup\set{\star}}$ can be interpolated by a $\frac{1-\mu}{1+\mu}$-contractive operator if and only if
\begin{align*}
    \ip{g_i - g_j, x_i -x_j}\geq 0\qquad\text{for all }i,j\in \cI_N\cup\set{\star}.
\end{align*}
\end{itemize}
\end{lemma}

Aggregating these inequalities gives the following result.

\begin{lemma}
    \label{lem:aggregated_interpolation}
Let $\tilde H\in \R^{\cI_N\times \cI_N}$ be strictly lower triangular.
Suppose $\Lambda\in \R^{\cI_N\times \cI_N}$ has nonnegative off-diagonal entries and $a, b\in\R^{\cI_N}$ are nonnegative and $\Lambda\bbm{1} = - a$.
In the operator setting, further assume that $\Lambda$ is symmetric and $a = b$.
Then,
\begin{align*}
        &\ip{\Lambda^\intercal\bbm{1}+b,\bf} +\bbm{1}^\intercal(a - b) f_\star
    - \tr(\bg^\intercal\Lambda\bx)\geq 0
\end{align*}
for any trajectory produced by $\tilde H$ on a $1$-smooth $\mu$-strongly convex function or $\frac{1-\mu}{1+\mu}$-contractive operator.
In the operator setting, the terms involving $\bf$ and $f_\star$ are identically zero and it is inconsequential that $\bf$ and $f_\star$ are undefined.
\end{lemma}
\begin{proof}
For $i,j\in \cI_N^\star \coloneqq \cI_N \cup\set{\star}$, define
\begin{align*}
    Q_{i,j} &\coloneqq f_j - f_i - \ip{g_i, x_j - x_i},
\end{align*}
where $f_i$ can be undefined in the operator setting.

In the function setting, $Q_{i,j}\geq 0$ for all $i,j$ with equality whenever $i=j$. In the operator setting,
$(Q_{i,j} + Q_{j,i})/2 = \ip{g_j - g_i, x_j - x_i}/2 \geq 0$ for all $i,j$ with equality whenever $i=j$.

We aggregate the $Q_{i,j}$ for $i,j\in \cI_N$ with weights $\Lambda_{i,j}$:
\begin{align*}
    \sum_{i,j \in \cI_N}\Lambda_{i,j} Q_{i,j}
    & =\ip{\Lambda^\intercal\bbm{1} - \Lambda\bbm{1}, \bf} - \ip{\Lambda -\Diag(\Lambda\bbm{1})  , \bg\bx^{\intercal}}.
\end{align*}
This is nonnegative in the function setting since $\Lambda$ has nonnegative off-diagonal entries. In the operator setting, we additionally use the assumption that $\Lambda$ is symmetric.

We aggregate $Q_{i,\star}$ over $i\in \cI_N$ with weights $a$ and $Q_{\star,i}$ over $i\in \cI_N$ with weights $b$:
\begin{align*}
    &\sum_{i\in \cI_N} a_i Q_{i,\star}+\sum_{i\in \cI_N} b_i Q_{\star,i}=\bbm{1}^\intercal a f_\star - \ip{a, \bf}
    + \ip{\Diag(a), \bg \bx^{\intercal}} + \ip{b,\bf} - \bbm{1}^\intercal b f_\star.
\end{align*}
Again, this is nonnegative in the function setting since $a,b\geq 0$. In the operator setting, we additionally use the assumption that $a=b$.
\end{proof}

\subsection{PEP certificates and rates}

The following definitions are inspired by the PEP literature.

\begin{definition}
Let $M_0,M_N\in\S^2_+$ be fixed PSD matrices.
Given $\tilde H\in\R^{\cI_N\times \cI_N}$ strictly lower triangular and $\tau > 0$, we say that $\Lambda\in\R^{\cI_N\times \cI_N}$ is a \emph{PEP certificate of $(M_0,M_N)$ convergence} of $\tilde H$ with rate $\tau$ if
$\Lambda$ has nonnegative off diagonal entries, nonpositive row sums, nonpositive column sums and
\begin{align*}
    \textup{Slack}^{M_0,M_N}_{\tau,\Lambda,\tilde H}(y_0, \bg) &\coloneqq 
    \frac{\tau}{2}\tr\left(\begin{bmatrix}
    y_0\\
    s_0
    \end{bmatrix}^\intercal M_0\begin{bmatrix}
    y_0\\
    s_0
    \end{bmatrix}\right) - \frac{1}{2\tau}\tr\left(\begin{bmatrix}
    y_N\\
    s_N
    \end{bmatrix}^\intercal M_N\begin{bmatrix}
    y_N\\
    s_N
    \end{bmatrix}\right)  + \tr(\bg^\intercal\Lambda\bx)
\end{align*}
is nonnegative for all $(y_0,\bg)$.
Here, $s_0, y_N, s_N, \bx$ are linear functions of $(y_0,\bg)$.

In the operator setting, $\Lambda$ is additionally required to be symmetric.
\end{definition}

We collect the sign constraints into their own definition. Define
\begin{align*}
    \cP_\textup{cert}\coloneqq\set{\Lambda:\, \begin{array}{l}
    \Lambda_{i,j}\geq 0 \quad\forall i\neq j\\
    \Lambda\bbm1 \leq 0\\
    \Lambda^\intercal\bbm1 \leq 0
    \end{array}},
\end{align*}
so that $\Lambda\in\cP_\textup{cert}$ if and only if it has nonnegative off-diagonal entries and nonpositive row and column sums.
The condition $\Lambda\in\cP_\textup{cert}$ is defined for square matrices $\Lambda$ of any dimension.

The standard performance criteria are
\begin{align*}
    M_D \coloneqq \begin{bmatrix}
    1 &\\
    & 0
    \end{bmatrix}\qquad M_G \coloneqq \begin{bmatrix}
    0 &\\
    & 1
    \end{bmatrix},
\end{align*}
corresponding to a distance-type criterion and gradient-type criterion respectively.

We will abbreviate
\begin{align*}
    Q_{M}(y_i,s_i)\coloneqq \tr\left(\begin{bmatrix}
    y_i\\
    s_i
    \end{bmatrix}^\intercal M \begin{bmatrix}
    y_i\\
    s_i
    \end{bmatrix}\right).
\end{align*}

For fixed performance criteria $M_0,M_N$ and a FSFOM $\tilde H$, we can define the PEP rate to be the infimal $\tau$ for which a PEP certificate exists.

\begin{definition}
Let $M_0,M_N\in\S^2_+$ be fixed PSD matrices and let $\tilde H\in\R^{\cI_N\times \cI_N}$ be strictly lower triangular.
Define the \emph{PEP rate}
\begin{align}
    \label{eq:pep_rate}
    \tau^{M_0,M_N}_\textup{PEP}(\tilde H)&\coloneqq \inf_{\substack{\Lambda\in\R^{\cI_N\times \cI_N}\\\tau>0}}\set{\tau:\, \begin{array}{l}
    \textup{Slack}_{\tau,\Lambda,\tilde H}^{M_0,M_N}(y_0,\bg)\geq 0\quad\forall (y_0,\bg)\\
    \Lambda\in \cP_\textup{cert}
    \end{array}
    }.
\end{align}
In the operator setting, $\Lambda$ is additionally constrained to be symmetric.
\end{definition}

The SDP appearing in the definition of $\tau^{M_0,M_N}_\textup{PEP}(\tilde H)$ is often referred to as the PEP dual problem. The corresponding PEP primal problem has $(y_0,\bg)$ as its variables and seeks the worst-case value of the ratio of the performance criteria subject to the interpolation constraints. The following lemma is a restatement of weak duality in this context.

\begin{lemma}
    \label{lem:cert_to_rate}
Suppose we are in either the function setting or the operator setting and fix $M_0,M_N\in\S^2_+$.
Let $\tilde H\in\R^{\cI_N\times \cI_N}$ be strictly lower triangular and let $\tau > 0$. Suppose there exists a PEP certificate $\Lambda\in\R^{\cI_N\times \cI_N}$ of $(M_0,M_N)$ convergence for $\tilde H$ with rate $\tau$.
In the operator setting, further suppose that $\Lambda$ is symmetric.
Then,
\begin{align}
    \label{eq:dd_guarantee}
    \frac{\tau}{2}Q_{M_0}(y_0,s_0) - \frac{1}{2\tau}Q_{M_N}(y_N,s_N) \geq 0 \qquad\text{for any trajectory produced by $\tilde H$.}
\end{align}
\end{lemma}
\begin{proof}
Fix such a $\Lambda$ and set $a = -\Lambda\bbm{1} \geq 0$ and $b = -\Lambda^\intercal\bbm{1} \geq 0$. In the operator setting, $\Lambda$ is additionally symmetric and $a = b$. Applying \cref{lem:aggregated_interpolation} gives
\begin{align*}
    -\tr(\bg^\intercal\Lambda\bx) \geq 0.
\end{align*}
Adding this inequality to
\begin{align*}
    \textup{Slack}^{M_0,M_N}_{\tau,\Lambda,\tilde H} =
    \frac{\tau}{2}Q_{M_0}(y_0,s_0) - \frac{1}{2\tau}Q_{M_N}(y_N,s_N) + \tr(\bg^\intercal\Lambda\bx),
\end{align*}
which is nonnegative by assumption, completes the proof.
\end{proof}

\begin{remark}
Under a high-dimensional assumption (usually $d\geq N+1$ or $d\geq N+2$) and a strong duality assumption (which holds under minor conditions), the converse of \cref{lem:cert_to_rate} also holds, i.e., 
$\tau_\textup{PEP}^{M_0,M_N}(\tilde H)$ is exactly the infimal $\tau$ for which \eqref{eq:dd_guarantee} holds.
Both assumptions are standard in the PEP literature. See, for example, \cite{taylor2017interpolation}.
\end{remark}

\section{A primal explanation: extremal trajectories}
\label{sec:extremal}
This section presents a conditional explanation for H-duality, assuming that the worst-case convergence behavior for a given FSFOM is witnessed by an ``extremal'' function or operator (to be defined below). We observe that this assumption holds for many optimized (and also unoptimized) FSFOMs in both the function and operator settings.
This argument is very general and will even allow us to predict nonstandard ``rate invariance'' relationships.


\subsection{Extremal trajectories}

In the function setting,
there are two natural extremal one-dimensional objects. Define
\begin{align*}
    \Omega_\textup{fun}\coloneqq\set{\mu,1}
\end{align*}
and for $\sigma\in \Omega_\textup{fun}$
define the univariate quadratic function
\begin{align*}
    h_\sigma(y) \coloneqq \frac{\sigma}{2}y^2.
\end{align*}
Then, $y_i$ and $s_i$ are related by $s_i = \grad h_\sigma(y_i) = \sigma y_i$.

The functions $h_\sigma$ are extremal in the sense that the interpolation constraints of \cref{lem:interpolation} are guaranteed to hold at equality. In fact, as we will show in \cref{lem:primal_dual_extremal}, any function for which the interpolation constraints hold exactly at equality can be decomposed into a direct sum of these $h_\sigma$.


The corresponding extremal objects in the operator setting additionally have a rotational component that naturally lives in two dimensions.\footnote{This rotational component comes from the fact that 
when $\Lambda$ is symmetric, the linear function encoded by $\Lambda$ is orthogonal to skew-symmetries. This will be apparent in the proof of \cref{lem:primal_dual_extremal}.} 
We identify $\R^2$ with $\C$. 
Define
\begin{align*}
    \Omega_\textup{op}\coloneqq \set{\sigma\in\C:\, \abs{\sigma - \frac{1+\mu}{2}} = \frac{1-\mu}{2}}.
\end{align*}
Note that $\Omega_\textup{fun}$ is the restriction of $\Omega_\textup{op}$ to the real line.
For $\sigma \in \Omega_\textup{op}$, define
\begin{align*}
    T_\sigma(y) \coloneqq \left(1 - \frac{2\sigma}{1+\mu}\right)y = \frac{1-\mu}{1+\mu}\left(\frac{\frac{1+\mu}{2}-\sigma}{\frac{1-\mu}{2}}\right)y.
\end{align*}
The expression on the right shows that $T_\sigma$ is a rotation in the complex plane followed by a contraction, so that $T_\sigma$ is $\frac{1-\mu}{1+\mu}$-contractive.
The corresponding scaled residual is
\begin{align*}
    s_i = \frac{1+\mu}{2}(y_i - T_\sigma (y_i)) = \sigma y_i.
\end{align*}
Again, the operators $T_\sigma$ are extremal in the sense that the interpolation constraints of \cref{lem:interpolation} are guaranteed to hold at equality and the proof of \cref{lem:primal_dual_extremal} will show that any such operator can be decomposed into a direct sum of these $T_\sigma$.







In the remainder of this section, we will simply write $\Omega$ with the understanding that it is $\Omega_\textup{fun}$ in the function setting and $\Omega_\textup{op}$ in the operator setting.
The above discussion shows that, on any of the extremal operators or functions, the pair $[y_i, s_i]$ lies on a line parameterized by $\sigma\in\Omega$. Formally, for $\sigma\in\Omega$, define
\begin{align*}
    d_\sigma \coloneqq \begin{bmatrix}
    1\\
    \sigma
    \end{bmatrix}.
\end{align*}
Then, $[y_i, s_i] = d_\sigma y_i$ for all $i\in \cI_N$.
Thus, on these extremal trajectories, it suffices to understand the behavior of the $y_i$.

\begin{definition}
Given $\sigma \in\Omega$ and a strictly lower triangular $\tilde H\in\R^{\cI_N\times \cI_N}$, define
\begin{align*}
    \kappa(\sigma,\tilde H)&\coloneqq e_N^\intercal (I + \sigma S\tilde H)^{-1}\bbm{1}.\qedhere
\end{align*}
\end{definition}

\begin{lemma}
Fix $\sigma\in\Omega$ and a FSFOM $\tilde H\in\R^{\cI_N\times \cI_N}$. Suppose $s_i = \sigma y_i$ for all $i\in \cI_N$. Then,
\begin{align*}
    y_N = \kappa(\sigma, \tilde H) y_0\qquad\text{and}\qquad
    s_N = \sigma \kappa(\sigma, \tilde H) y_0.
\end{align*}
\end{lemma}
\begin{proof}
It holds that
\begin{align*}
    \by = -S\tilde H \bs + \bbm{1} y_0 = -\sigma S\tilde H \by + \bbm{1} y_0.
\end{align*}
Solving for $\by$ and taking its $N$th component gives
    $y_N = e_N^\intercal (I + \sigma S\tilde H)^{-1}\bbm{1} y_0 = \kappa(\sigma, \tilde H) y_0$.
\end{proof}

\subsection{H-duality for extremal methods}

\begin{definition}
    \label{def:extremal_method}
    Fix performance criteria $M_0,M_N\in\S^2_+$.
We say that $\tilde H$ is an \emph{extremal method} if 
\eqref{eq:dd_guarantee} holds at equality for one of the extremal trajectories with $y_0\neq 0$ and $\tau =\tau_\textup{PEP}^{M_0,M_N}(\tilde H)$.
\end{definition}

The following lemma explicitly computes the PEP rate for an extremal method $\tilde H$.
\begin{lemma}
    \label{lem:rate_for_extremal}
Suppose $\tilde H \in\R^{\cI_N\times \cI_N}$ is an extremal method for $(M_0,M_N)$. 
Furthermore, suppose $d_\sigma^* M_0 d_\sigma\neq 0$ for all $\sigma\in\Omega$.
Then,
\begin{align*}
    \tau_\textup{PEP}^{M_0,M_N}(\tilde H) = \max_{\sigma\in\Omega}\abs{\kappa(\sigma, \tilde H)}\sqrt{\frac{d_\sigma^* M_N d_\sigma}{d_\sigma^* M_0 d_\sigma}}.
\end{align*}
\end{lemma}
\begin{proof}
First, note that for $n\in\set{0,N}$, we have $Q_{M_n}(y_n, s_n) = \tr\left(y_n^* d_\sigma^* M_n d_\sigma y_n\right) = \norm{y_n}^2 d_\sigma^* M_n d_\sigma$.

Then, by definition of an extremal method
\begin{align*}
    \tau_\textup{PEP}^{M_0,M_N}(\tilde H) &= \max_{\sigma \in \Omega}\sqrt{\frac{Q_{M_N}(y_N,s_N)}{Q_{M_0}(y_0,s_0)}}
    = \max_{\sigma \in \Omega}\frac{\norm{y_N}}{\norm{y_0}}\sqrt{\frac{d_\sigma^* M_N d_\sigma}{d_\sigma^* M_0 d_\sigma}}
    = \max_{\sigma \in \Omega}\abs{\kappa(\sigma,\tilde H)}\sqrt{\frac{d_\sigma^* M_N d_\sigma}{d_\sigma^* M_0 d_\sigma}}.\qedhere
\end{align*}
\end{proof}

The following lemma states that $\kappa(\sigma,\tilde H)$ is invariant under antitransposition of $\tilde H$. We view this as the key mechanism underlying H-duality for extremal methods.

\begin{lemma}
\label{lem:extremal_kappa_invariance}
Suppose $\sigma \in\C$ and $\tilde H$ is strictly lower triangular. Then,
\begin{align*}
    \kappa(\sigma,\tilde H) = \kappa(\sigma,\tilde H^A).
\end{align*}
\end{lemma}
\begin{proof}
We compute
\begin{align*}
    \kappa(\sigma,\tilde H) &= e_N^\intercal(I + \sigma S \tilde H)^{-1}\bbm{1}\\
    &= \bbm{1}^\intercal(I + \sigma \tilde H^\intercal S^\intercal)^{-1}e_N\\
    &= \bbm{1}^\intercal(I + \sigma P^{-1}S\tilde H^AP)^{-1}e_N\\
    &= \bbm{1}^\intercal P^{-1}(I + \sigma S\tilde H^A)^{-1} Pe_N\\
    &= e_N^\intercal (I + \sigma S\tilde H^A)^{-1} \bbm1\\
    &= \kappa(\sigma, \tilde H^A).\qedhere
\end{align*}
\end{proof}

H-duality for extremal methods then follows as a corollary for any pair of performance criteria satisfying a single identity.

\begin{theorem}[H-Duality for Extremal Methods]
    \label{thm:h_duality_for_extremal}
Fix a pair of performance criteria $(M_0,M_N)$ and $(M_0', M_N')$. Let $\tilde H\in\R^{\cI_N\times \cI_N}$ and suppose $\tilde H$ is extremal for $(M_0,M_N)$ and $\tilde H^A$ is extremal for $(M_0',M_N')$.
Furthermore, suppose
\begin{align*}
    \frac{d_\sigma^* M_N d_\sigma}{d_\sigma^* M_0 d_\sigma} = \frac{d_\sigma^* M_N' d_\sigma}{d_\sigma^* M_0' d_\sigma}\qquad\forall \sigma \in\Omega
\end{align*}
and that the denominators are nonzero for all $\sigma\in\Omega$.
Then,
\begin{align*}
    \tau_\textup{PEP}^{M_0,M_N}(\tilde H) = \tau_\textup{PEP}^{M_0',M_N'}(\tilde H^A).
\end{align*}
\end{theorem}

\begin{corollary}
Fix $\tilde H\in\R^{\cI_N\times \cI_N}$ and $q\geq 0$.
In the list below, if $\tilde H$ is extremal for the performance criteria on the left and $\tilde H^A$ is extremal for the performance criteria on the right, then the PEP rates are equal:
\begin{align*}
    \tau_\textup{PEP}^\textup{D,G}(\tilde H) &= \tau_\textup{PEP}^\textup{D,G}(\tilde H^A)\\
    \tau_\textup{PEP}^\textup{D,D}(\tilde H) &= \tau_\textup{PEP}^\textup{G,G}(\tilde H^A)\\
    \tau_\textup{PEP}^\textup{G,D}(\tilde H)&= \tau_\textup{PEP}^\textup{G,D}(\tilde H^A).
\end{align*}
The last two dualities additionally require $q>0$.
\end{corollary}
\begin{proof}
By \cref{thm:h_duality_for_extremal}, it suffices to check that for all $\sigma$,
\begin{align*}
\frac{d_\sigma^* M_N d_\sigma}{d_\sigma^* M_0 d_\sigma} = \frac{d_\sigma^* M_N' d_\sigma}{d_\sigma^* M_0' d_\sigma}.
\end{align*}

This is tautologically true in the (D,G)--(D,G) and (G,D)--(G,D) settings.
In the (D,D)--(G,G) setting,
\begin{align*}
    \frac{d_\sigma^* M_D d_\sigma}{d_\sigma^* M_D d_\sigma} = 
    \frac{1}{1} = \frac{\abs{\sigma}^2}{\abs{\sigma}^2} = \frac{d_\sigma^* M_G d_\sigma}{d_\sigma^* M_G d_\sigma}.
\end{align*}
The requirement $q>0$ ensures that $d_\sigma^* M_G d_\sigma > 0$ for all $\sigma\in\Omega$.
\end{proof}

\begin{remark}
\cref{thm:h_duality_for_extremal} 
not only explains the standard dualities: (D,D)--(G,G), (D,G)--(D,G), and (G,D)--(G,D), but further \emph{predicts} more general ``rate invariance'' or ``rate transfer'' phenomena. For example, consider the performance criteria
\begin{gather*}
    M_0 = \begin{bmatrix}
    1/2\\
    & 1/2
    \end{bmatrix}\qquad M_N = \begin{bmatrix}
    1\\
    & 0
    \end{bmatrix}\\
    M_0' = \begin{bmatrix}
    \rho-\mu\\
    1-\rho
    \end{bmatrix}\begin{bmatrix}
    \rho-\mu\\
    1-\rho
    \end{bmatrix}^\intercal \qquad M_N' = \begin{bmatrix}
    1+\mu\\
    -2
    \end{bmatrix}\begin{bmatrix}
    1+\mu\\
    -2
    \end{bmatrix}^\intercal
\end{gather*}
where $\rho = \sqrt{(1+\mu^2)/2}$.
These performance criteria satisfy the assumptions of \cref{thm:h_duality_for_extremal} over $\sigma \in \Omega_\textup{op}$ and correspond to the guarantees
\begin{gather*}
    \frac{\tau}{2} \left(\frac{1}{2}\norm{y_0}^2 + \frac{1}{2}\norm{s_0}^2\right) \geq \frac{1}{2\tau}\norm{y_N}^2\quad\text{and}\\
    \frac{\tau}{2}\norm{(\rho-\mu)y_0 + (1-\rho)s_0}^2 \geq \frac{1}{2\tau}\norm{(1+\mu)y_N - 2 s_N}^2
\end{gather*}
respectively. 
In the operator setting, the right-hand side on the second line is
$\frac{(1+\mu)^2}{2\tau}\norm{T(y_N)}^2$, thus this gives a guarantee on the distance between the final oracle output $T(y_N)$ and $y_\star$.
\cref{thm:h_duality_for_extremal} then implies that H-duality holds between these performance criteria whenever $\tilde H$ and $\tilde H^A$ are extremal in the two settings.
We emphasize that although the output metric $M_N=M_D$ is the distance-type criterion, $M_0'$ is \emph{not} the gradient-type criterion $M_G$.

One can verify numerically whether this predicts a genuine H-duality.
For example, taking $N = 5$, $q=0.1$, and $\tilde H$ sampled uniformly from a cube centered at gradient descent\footnote{Specifically, $\tilde H$ is strictly lower triangular with independent entries $\tilde H_{ij}$ drawn uniformly from $[0.9,1.1]$ for the entries on the first subdiagonal and uniformly from $[-0.1,0.1]$ everywhere else. Numerically, we declare that $\tau_\textup{PEP}^{M_0,M_N}(\tilde H) = \tau_\textup{PEP}^{M_0',M_N'}(\tilde H^A)$ if their relative error is less than $10^{-5}$. See \url{https://github.com/alexlihengwang/H-Duality} for experimentation details.}, 
we observe that $\tau_\textup{PEP}^{M_0,M_N}(\tilde H) = \tau_\textup{PEP}^{M_0',M_N'}(\tilde H^A)$
in 1000/1000 trials in both the function and operator settings.
This experiment indicates that these performance setups satisfy an H-duality on a full-dimensional set of $\tilde H$.
\end{remark}
\begin{remark}
    Note that in the function setting, the condition that
    \begin{align*}
        \frac{d_\sigma^* M_N d_\sigma}{d_\sigma^* M_0 d_\sigma} = \frac{d_\sigma^* M_N' d_\sigma}{d_\sigma^* M_0' d_\sigma}\qquad\forall \sigma \in\Omega
    \end{align*}
    reduces to the two conditions that 
    \begin{align*}
        \frac{d_\mu^{\intercal} M_N d_\mu}{d_\mu^{\intercal} M_0 d_\mu} = \frac{d_\mu^{\intercal} M_N' d_\mu}{d_\mu^{\intercal} M_0' d_\mu}\qquad\text{and}\qquad
        \frac{d_1^{\intercal} M_N d_1}{d_1^{\intercal} M_0 d_1} = \frac{d_1^{\intercal} M_N' d_1}{d_1^{\intercal} M_0' d_1}.
    \end{align*}

    On the other hand, in the operator setting, this condition becomes somewhat more interesting, as $\Omega_\textup{op}$ is now an infinite set of points.
    In this setting, the condition on the performance setups amounts to the equation of real polynomials (identifying $\C=\R^2$) on the circle $\Omega\subseteq\R^2$:
    \begin{align*}
        (d_\sigma^* M_N d_\sigma)(d_\sigma^* M_0' d_\sigma) = (d_\sigma^* M_0 d_\sigma)(d_\sigma^* M_N' d_\sigma)\qquad\forall \sigma \in\Omega.
    \end{align*}
    From this, we can arrive at the conclusion that the condition holds if and only if there exists some constant $c \in \R$ so that either
    \begin{align*}
        d_\sigma^* M_N d_\sigma = c d_\sigma^* M_N' d_\sigma \text{ and } d_\sigma^* M_0 d_\sigma = cd_\sigma^* M_0' d_\sigma\qquad\forall \sigma \in\Omega,
    \end{align*}
    or it holds that 
    \begin{align*}
        d_\sigma^* M_N d_\sigma &= c d_\sigma^* M_0 d_\sigma \text{ and }d_\sigma^* M_N' d_\sigma = c d_\sigma^* M_0' d_\sigma\qquad\forall \sigma \in\Omega.\qedhere
    \end{align*}
\end{remark}

\section{A dual explanation: certificate transformations}
\label{sec:congruences}
We will now turn to a \emph{certificate-level} explanation of H-duality for the dualities (D,G)--(D,G), (D,D)--(G,G), and (G,D)--(G,D). To be explicit, we will define a transformation of certificate matrices which preserves many of the properties of a certificate, and in many cases actually produces a certificate for the fact that the algorithm and its H-dual have the same performance characteristics.

In order to give a complete explanation of our transformation for certificates, we will need to introduce some new notions of certificates to simplify the exposition.
One issue that we will need to confront is that for several of the classical performance metrics, the optimal PEP certificates are often singular, i.e. they have nontrivial kernels.
In order to remove this issue, in \cref{subsec:reduced}, we will introduce the notion of `reduced' certificates in each of these sections that remove the issues of the kernels.

The other issue we will need to confront is that it is in general more difficult to preserve the sign conditions of a PEP certificate across the H-duality transformation.
Recall that a PEP certificate $\Lambda$ is required to satisfy a nonlinear SDP constraint, $\textup{Slack}^{M_0,M_N}_{\tau,\Lambda,\tilde H}(y_0,\bg)\geq 0$ everywhere, and linear constraints, $\Lambda\in \cP_\textup{cert}$.
\cref{subsec:extremal_certs} 
defines an \emph{unsigned certificate} $\Lambda$ to be one that satisfies the SDP constraint but not necessarily the linear constraints.
This definition is motivated by the extremal-methods explanation of H-duality, as formalized in \cref{subsec:extremal_certs}, and separates the PSD congruence question from a sign-preservation question.
\cref{subsec:lifted_certs} constructs lifted matrices $\cL^{M_0,M_N}_{\Lambda,\tilde H}$ such that
$\textup{Slack}^{M_0,M_N}_{\tau,\Lambda,\tilde H}(y_0,\bg)\geq 0$ everywhere if and only if $\cL^{M_0,M_N}_{\Lambda,\tilde H}\succeq 0$.

Then,
\cref{subsec:DD_duality,subsec:DG_duality,subsec:GD_duality} treat each of the dualities in~\eqref{eq:classical_dualities} and, in each setting, shows that for an appropriate notion of a reduced unsigned certificate $\Lambda$, the matrix $(P\Lambda P)^{-1}$ is a reduced unsigned certificate in the H-dual setting, i.e. if $\Phi = (P\Lambda P)^{-1}$, we have that
\begin{align*}
    \cL^{M_0,M_N}_{\Lambda,\tilde H} \equiv \cL^{M_0',M_N'}_{\Phi,\tilde H^A},
\end{align*}
where $\equiv$ denotes congruence.
This congruence holds for every $\Lambda \in \R^{\mathcal{I}_N \times \mathcal{I}_N}$, even if the slack is not everywhere nonnegative, or if $\Lambda$ is not in $\cP_{cert}$.
Consequently, the map sends an unsigned certificate for $\tilde H$ with specified criteria and rate $\tau$ to an unsigned certificate for $\tilde H^A$ with dual criteria and the same rate.

We then specialize to the case in which $\Lambda$ has a specific sparsity pattern, namely that it is tridiagonal. This has been a popular subject of study because many optimal algorithms have tridiagonal sparsity patterns. We give a simple unified proof that if $\Lambda$ is a reduced tridiagonal PEP certificate for an algorithm and one of the performance setups in \eqref{eq:classical_dualities}, then in fact, $(P\Lambda P)^{-1}$ is a reduced PEP certificate for the H-dual algorithm. In particular, this shows that the H-dual of ITEM, a minimax optimal algorithm for function minimization in the (D,D) setting, has the same rate of convergence in the (G,G) setting. This is the first time that the rate of convergence of the H-dual algorithm (which we refer to as ITEM-GG) has been rigorously analyzed, though its convergence had been analyzed numerically previously.

\subsection{Unsigned certificates}
\label{subsec:extremal_certs}

The definition of an extremal method presented in \cref{def:extremal_method} can be thought of as a ``primal'' definition since it is stated in terms of worst-case functions. The following definitions can be thought of as the corresponding ``dual'' definitions.

\begin{definition}
\label{def:extremal_cert}
    Fix $M_0,M_N\in\S^2_+$.
    Let $\tilde H\in\R^{\cI_N\times \cI_N}$ be strictly lower triangular and let $\tau > 0$. 
    
    \begin{itemize}
        \item We say that $\Lambda\in\R^{\cI_N\times \cI_N}$ is an \emph{unsigned certificate} of $\tilde H$ with rate $\tau$ if
    $\textup{Slack}^{M_0,M_N}_{\tau,\Lambda,\tilde H}(y_0, \bg)$
    is nonnegative for all $(y_0,\bg)$.
    \item We define the \emph{unsigned rate} to be the infimal $\tau$ for which an unsigned certificate exists:
    \begin{align}
    \label{eq:tau_extremal_general}
        \tau^{M_0,M_N}_\textup{unsigned}(\tilde H)&\coloneqq \inf_{\substack{\Lambda\in\R^{\cI_N\times \cI_N}\\\tau>0}}\set{\tau:\, 
        \textup{Slack}_{\tau,\Lambda,\tilde H}^{M_0,M_N}(y_0,\bg)\geq 0\quad\forall (y_0,\bg)
        }.
    \end{align}
    In the operator setting, we additionally require $\Lambda$ symmetric.\qedhere
    \end{itemize}
\end{definition}

Compared to the PEP rate $\tau_\textup{PEP}^{M_0,M_N}(\tilde H)$ defined in \eqref{eq:pep_rate}, the SDP above drops the sign constraints encoded by $\cP_\textup{cert}$.
The corresponding primal SDP (worst-case instance search) then requires that the interpolation conditions hold at \emph{equality} between every query point and optimizer. Intuitively, the feasible trajectories will be exactly direct sums of the extremal trajectories considered in the previous section.
The following proposition formalizes this intuition and states
that, under a strong duality assumption, $\tau_\textup{PEP}^{M_0,M_N}(\tilde H)=\tau_\textup{unsigned}^{M_0,M_N}(\tilde H)$ if and only if $\tilde H$ is extremal.
Its proof is orthogonal to the main thrust of this section and deferred to \cref{app:deferred}.

\begin{proposition}
    \label{lem:primal_dual_extremal}
    Fix $M_0,M_N\in \S^2_+$ and suppose $\tilde H \in\R^{\cI_N\times \cI_N}$ is strictly lower triangular. 
    Suppose $d_\sigma^* M_0 d_\sigma >0$ for all $\sigma\in\Omega$.
    Suppose that strong duality holds between the SDP defining $\tau_\textup{unsigned}^{M_0,M_N}(\tilde H)$ and its dual. Then, $\tilde H$ is extremal if and only if
\begin{align*}
    \tau_\textup{PEP}^{M_0,M_N}(\tilde H) =
    \tau_\textup{unsigned}^{M_0,M_N}(\tilde H).
\end{align*}
\end{proposition}

\subsection{Lifted slack matrices}
\label{subsec:lifted_certs}

Rather than working with the slacks directly, we will lift them.



\begin{lemma}
    \label{lem:certificateLift}
    Let $q\geq 0$.
Fix $\tilde H$ strictly lower triangular and
$(M_0, M_N)\in\set{M_D,M_G}^2$. 
If $M_0 = M_G$, assume additionally that $q>0$.
Then $\Lambda$ is a PEP certificate for $(M_0, M_N)$ convergence with rate $\tau$ if and only if $\Lambda\in\cP_{\textup{cert}}$ and
\begin{align}
    \label{eq:lifted_slack}
    \cL_{\Lambda, \tilde H}^{M_0, M_N} \coloneqq \begin{bmatrix}
    -2\sym(\Lambda A) & \Lambda b & w\\
    \cdot & \tau & \zeta\\
    \cdot & \cdot & \tau
    \end{bmatrix}\succeq 0,
\end{align}
where we define
\begin{align*}
    &c_0 \coloneqq \begin{cases}
    1+q &\text{if }M_0 = M_D\\
    q&\text{if }M_0 = M_G
    \end{cases}&&
    c_N \coloneqq \begin{cases}
    1+q &\text{if }M_N = M_D\\
    q&\text{if }M_N = M_G
    \end{cases}\\
    &\theta\coloneqq\frac{1+q-c_0}{c_0}
    &&\eta_H \coloneqq  e_N^\intercal v_H
    &&A  \coloneqq SH + \theta v_H e_0^\intercal\\
    &b \coloneqq (1+\theta)v_H
    &&w \coloneqq (I- c_N SH)^\intercal e_N - c_N \theta \eta_H e_0
    &&\zeta \coloneqq c_N (1 +\theta) \eta_H.
\end{align*}
\end{lemma}
\begin{proof}
Within this proof, define
\begin{align*}
    p_i &\coloneqq g_i + c_i x_i = (c_i - q)y_i + (1 + q - c_i) s_i\qquad\text{for $i=0,N$}.
\end{align*}
It holds that
\begin{align*}
    p_0 = \begin{cases}
    y_0&\text{if }M_0 = M_D\\
    s_0&\text{if }M_0 = M_G
    \end{cases}\qquad\text{and}\qquad
    p_N = \begin{cases}
    y_N&\text{if }M_N = M_D\\
    s_N&\text{if }M_N = M_G
    \end{cases}.
\end{align*}

By definition, $\Lambda$ is a PEP certificate for $(M_0, M_N)$ convergence with rate $\tau$ if and only if $\Lambda\in\cP_{\textup{cert}}$ and
\begin{align}
    \label{eq:lifting_original_condition}
    \frac{\tau}{2}\norm{p_{0}}^2 - \frac{1}{2\tau}\norm{p_{N}}^2 + \tr(\bg^\intercal \Lambda \bx )\geq 0 \qquad\text{for all $(y_0,\bg)$}.
\end{align}
It suffices to verify that \eqref{eq:lifting_original_condition} holds if and only if \eqref{eq:lifted_slack}.

As a first step, introduce an auxiliary variable $\epsilon\in\R^d$ and note that \eqref{eq:lifting_original_condition} holds if and only if 
\begin{align}
    \label{eq:lifting_with_aux}
    \frac{\tau}{2}\norm{p_{0}}^2 + \frac{\tau}{2}\norm{\epsilon}^2 + \ip{\epsilon,p_{N}} + \tr(\bg^\intercal \Lambda \bx )\geq 0 \qquad\text{for all $(y_0,\bg,\epsilon)$}.
\end{align}
This follows because $\min_{\epsilon\in\R^d}\frac{\tau}{2}\norm{\epsilon}^2 + \ip{\epsilon,p_{N}} = - \frac{1}{2\tau}\|p_N\|^2$.

We claim that there is an invertible transformation
$(y_0,\bg,\epsilon)\leftrightarrow (\bg, p_0, \epsilon)$. Indeed,
\begin{align*}
    y_0 = (1+q) x_0 + g_0 = \frac{1+q}{c_0}(p_0-g_0) + g_0,
\end{align*}
where $c_0 >0$ by assumption. Thus, we may quantify \eqref{eq:lifting_with_aux} over $(\bg,p_0,\epsilon)$ instead.

It remains to write $\bx$ and $p_N$ as linear functions in $(\bg, p_0,\epsilon)$:
\begin{align*}
    \bx &= -SH \bg + v_H y_0
     = -\left(SH + \theta v_H e_0^\intercal\right) \bg + (1+\theta)v_H p_{0}
    =- A \bg + b p_{0}.
\end{align*}

Similarly,
\begin{align*}
    p_N &= g_N + c_N x_N 
    = e_N^\intercal(I - c_N SH) \bg + \eta_H c_N y_0
    = e_N^\intercal(I - c_N SH) \bg - \eta_H c_N \theta g_0 + \eta_H c_N (1+\theta)p_{0}
    = w^\intercal \bg + \zeta p_{0}.
\end{align*}

Rewriting the LHS of \eqref{eq:lifting_with_aux} as
\begin{align*}
    \frac{1}{2}\tr\left(\begin{bmatrix}
    \bg\\
    p_0\\
    \epsilon
    \end{bmatrix}^\intercal \begin{bmatrix}
    -2\sym(\Lambda A) & \Lambda b & w\\
    \cdot & \tau & \zeta\\
    \cdot & \cdot & \tau
    \end{bmatrix}\begin{bmatrix}
    \bg\\
    p_0\\
    \epsilon
    \end{bmatrix}\right)
\end{align*}
completes the proof.
\end{proof}

\subsection{Distance--Gradient self-duality}
\label{subsec:DG_duality}

This section treats the (D,G)--(D,G) duality. This is the simplest of the three dualities in \eqref{eq:classical_dualities} and serves as a model for the other two dualities, which are similar in spirit but algebraically denser.


We will repeatedly use the following identity in all three of our duality proofs. It is a special case of \cref{lem:extremal_kappa_invariance}, which was our ``key mechanism'' underlying H-duality in \cref{sec:extremal}.

\begin{lemma}
\label{lem:eN_vH}
Fix an arbitrary $\tilde H$ strictly lower triangular. Then,
    $e_N^\intercal v_{H} = e_N^\intercal v_{H^A}$.
\end{lemma}
\begin{proof}
The proof of \cref{lem:x_identity} shows that
\begin{align*}
    I - q SH = ((1+q)I + q S\tilde H)^{-1} = \frac{1}{1+q}\left(I + \mu S\tilde H\right)^{-1}.
\end{align*}
Thus,
\begin{align*}
    e_N^\intercal v_H &= e_N^\intercal (I - q SH)\bbm1 = \frac{1}{1+q}e_N^\intercal \left(I + \mu S\tilde H\right)^{-1}\bbm1 = \frac{1}{1+q}\kappa(\mu,\tilde H).
\end{align*}
The result now follows from \cref{lem:extremal_kappa_invariance}.
\end{proof}

\begin{lemma}
    \label{lem:dg_block_calculations}
    Let $\tilde H$ strictly lower triangular. Define $A, b, w, \zeta$ as in \cref{lem:certificateLift} with respect to $\tilde H, (M_D, M_G)$.
Define $A^\vee, b^\vee, w^\vee, \zeta^\vee$ analogously as in \cref{lem:certificateLift} with respect to $\tilde H^A, (M_D, M_G)$. Then,
\begin{align*}
    &P^{-1} A^\vee P = A^\intercal
    &&P w^\vee = b
    &&P^{-1} b^\vee = w
    &&\zeta^\vee = \zeta.
\end{align*}
\end{lemma}
\begin{proof}
Note that $\theta = \theta^\vee = 0$.
The first identity follows as
\begin{align*}
    P^{-1} A^\vee P &= P^{-1}SH^AP= H^\intercal S^\intercal= A^\intercal.
\end{align*}
The second identity follows as
\begin{align*}
    P w^\vee = P (I - q SH^A)^\intercal e_N = \bbm 1 - q PH^{A\intercal} S^\intercal e_N = (I - q S H)  \bbm 1 = v_H.
\end{align*}
The third identity follows as
\begin{align*}
    P^{-1} b^\vee = P^{-1} v_{H^A} = P^{-1}(I - q SH^A)Pe_N = (I - q P^{-1}SH^A P)e_N = (I - q SH)^\intercal e_N = w.
\end{align*}
The final identity is
\begin{align*}
    \zeta^\vee &= q\eta_{H^A} = q \eta_H = \zeta.\qedhere
\end{align*}
\end{proof}


\begin{theorem}
    \label{thm:DG}
    Let $\tilde H\in\R^{\cI_N\times \cI_N}$ be strictly lower triangular and let $\tau>0$. Let $\Lambda\in\R^{\cI_N\times \cI_N}$ be invertible and define
    \begin{align*}
        \Phi \coloneqq (P\Lambda P)^{-1}\qquad\text{and}\qquad
        \cR \coloneqq \begin{bmatrix}
        \Lambda P\\
        & 0 & 1\\
         & 1 & 0
        \end{bmatrix}
    \end{align*}
    Then,
    \begin{align*}
        \cL_{\Lambda,\tilde H}^\textup{D,G} = \cR \cL_{\Phi,\tilde H^A}^\textup{D,G}\cR^\intercal.
    \end{align*}
    It follows that $\Lambda$ is an unsigned certificate for (D,G) convergence of $\tilde H$ with rate $\tau$ if and only if $\Phi$ is an unsigned certificate for (D,G) convergence of $\tilde H^A$ with rate $\tau$.
    Moreover, if $\Lambda$ is symmetric, then $\Phi$ is also symmetric.
\end{theorem}
\begin{proof}
The congruence claim follows immediately from \cref{lem:dg_block_calculations}.
Since $\cR$ is invertible, this shows that $\cL^\textup{D,G}_{\Lambda,\tilde H}$ and $\cL^\textup{D,G}_{\Phi,\tilde H^A}$ are congruent.
Since $P$ is symmetric, $\Lambda$ being symmetric implies $\Phi$ is symmetric.
\end{proof}

H-duality follows as an immediate corollary whenever the sign conditions are preserved.
\begin{corollary}
    \label{cor:DG_convergence}
    Let $\tilde H$ be strictly lower triangular and let $\tau >0$. Let $\Lambda\in\R^{\cI_N\times \cI_N}$ be a PEP certificate for (D,G) convergence of $\tilde H$ with rate $\tau$ and suppose $\Lambda$ is invertible. If $\Phi\coloneqq (P \Lambda P)^{-1}$ satisfies the linear conditions $\Phi\in\cP_\textup{cert}$, then $\Phi$ is a PEP certificate for (D,G) convergence of $\tilde H^A$ with rate $\tau$.
\end{corollary}

\begin{remark}
    \label{rem:beyond_tridiagonal}
    \cref{cor:DG_convergence} applies to any $\Lambda$ for which $\Phi\coloneqq (P\Lambda P)^{-1}$ satisfies $\Phi\in\cP_\textup{cert}$.
    The following experiment indicates that this property is not confined to $\Lambda$ with a particular support:
    We randomly sampled 100 $\tilde H$ matrices in the function and operator settings with $N=5$ and $q=0.1$ for which the (D,G) rates of $\tilde H$ and $\tilde H^A$ matched.\footnote{$\tilde H$ is sampled with i.i.d.\ entries uniform on $[0,2]$, $[-1/2,1/2]$, $[-1/4,1/4]$, $[-1/8,1/8]$, $[-1/16,1/16]$ on the 1st-5th subdiagonals. We reject the sample if the PEP rates for $\tilde H$ and $\tilde H^A$ differ by more than a relative error of $10^{-7}$. See \url{https://github.com/alexlihengwang/H-Duality} for implementation details.} For each pair, we ran a simple local search for PEP certificates $\Lambda$ and $\Phi$ for $\tilde H$ and $\tilde H^A$ respectively minimizing the error in $P\Lambda P \Phi = I$.
    The certificates found in this way had entrywise magnitude $P\Lambda P \Phi - I$ at most $10^{-10}$ in 99/100 instances in the function setting and 96/100 instances in the operator setting.

One representative trial in the function setting finds:
\begin{align*}
    \tilde H &= \begin{bmatrix}
     0.0   &   0.0  &    0.0   &  0.0  &   0.0  &    0.0\\
  0.2073  & 0.0   &  0.0   &  0.0 &    0.0  &    0.0\\
 -0.3162  & 0.1829  & 0.0   &   0.0  &    0.0  &    0.0\\
  0.012   & 0.0489 &  0.5554  &0.0  &   0.0  &    0.0\\
  0.0511  & 0.2496  & 0.1043 &  0.8162  & 0.0  &   0.0\\
 -0.0371  &-0.0501 & -0.239  &  0.1246 &  1.7041&  0.0
    \end{bmatrix}\\
    \Lambda &= \begin{bmatrix}
     -0.2804  & 0.0405 &  0.0087 &  0.03  &   0.0249 &  0.0017\\
  0.2787 & -1.0491 &  0.5149 &  0.0483 &  0.093  &  0.0033\\
  0.0008 &  0.6477 & -1.1367 &  0.1696 &  0.3102 &  0.0023\\
  0.0001 &  0.0065 &  0.5745 & -1.1818 &  0.5472  & 0.0042\\
  0.0001 &  0.1346 &  0.0049 &  0.8524 & -2.1216 &  0.9911\\
  0.0002 &  0.006  &  0.0197 &  0.0077  & 1.1048  &-1.5969
    \end{bmatrix}\\
    \Phi &= \begin{bmatrix}
     -0.7952  & 0.0663 &  0.0748  & 0.0504  & 0.0006 &  0.2594\\
  0.1121 & -1.0145  & 0.2342 &  0.2113 &  0.0004  & 0.2823\\
  0.0603  & 0.6891 & -1.4674 &  0.2489 &  0.0007 &  0.3693\\
  0.092  &  0.0224  & 0.554   &-1.2496 &  0.005 &   0.4243\\
  0.1544 &  0.0181 &  0.4416  & 0.5946 & -3.5725 &  2.1029\\
  0.157   & 0.089  &  0.0352  & 0.0063 &  3.5655 & -4.13
    \end{bmatrix}
\end{align*}
Note that $\Lambda$ and $\Phi$ are dense.
\end{remark}

\subsection{Reduced certificate space and restricted congruences}
\label{subsec:reduced}
The (D,D)--(G,G) and (G,D)--(G,D) H-duality proofs are similar in spirit but require heavier algebra.
There are two main reasons for this.
First, the PEP certificates in these settings naturally live in lower-dimensional $N\times N$ dimensional subspaces of $\R^{\cI_N\times \cI_N}$. These subspaces are also where the H-dualities naturally act. For example, in the (D,D) and (G,D) settings, the final gradient $s_N$ (residual in the operator setting) has no effect on the convergence guarantee so that a PEP certificate should only require nonzero $\Lambda_{ij}$ in the top-left $N\times N$ block.
Second, the lifted slacks in these settings also naturally have kernels. In particular, the congruence transformation $\cR$ naturally only acts on the orthogonal complement to these kernels.

Let $\bar S, \bar R,\bar P\in\R^{N\times N}$ denote the $N\times N$ analogues of $S,R,P\in\R^{I_N\times I_N}$.
Define the following matrices mapping $\R^N\to \R^{\cI_N}$ 
\begin{align*}
   U = \begin{bmatrix}
   -\bbm{1}_N^\intercal\\
   I_N
   \end{bmatrix},\qquad
   V = \begin{bmatrix}
   I_N\\
   0^\intercal
   \end{bmatrix},\qquad
    W = \begin{bmatrix}
   \bar R\\
   0^\intercal
   \end{bmatrix},
\end{align*}
where $\bbm1_N\in\R^N$ and $I_N\in\R^{N\times N}$ are the all-ones vector and identity matrix of the appropriate dimensions.

Given $\Lambda\in\R^{N\times N}$, define the embeddings of $\Lambda$ into $\R^{\cI_N\times \cI_N}$,
\begin{align*}
    \Lambda_{U} = U\Lambda U^\intercal,\qquad
    \Lambda_{V} = V\Lambda V^\intercal,\qquad
    \Lambda_{W} = W\Lambda W^\intercal.
\end{align*}
We will say that $\Lambda$ is a \emph{reduced certificate} for (D,D) convergence if $\Lambda_V$ is a PEP certificate for (D,D) convergence, and similarly we say that $\Lambda$ is an unsigned reduced certificate for (D,D) convergence if $\Lambda_V$ is an unsigned certificate for (D,D) convergence. We make analogous definitions for reduced certificates of (G,D) convergence with the embedding $\Lambda_U$, and reduced certificates of (G,G) convergence with the embedding $\Lambda_W$.

We will write $\bar P_W^{-1} \coloneqq (\bar P^{-1})_W$. Note that $(\bar P_W)^{-1}$ is not defined, so there is no chance of confusion.

The following lemma explains that these embeddings are the natural spaces where certificates live in the various settings.
We defer its proof to \cref{app:deferred}.
\begin{lemma}
    \label{lem:dd_gg_gd_certificate_form}
    Let $\tilde H\in\R^{\cI_N\times \cI_N}$ be strictly lower triangular and let $\tau >0$. If $\Lambda \in\R^{\cI_N\times \cI_N}$ is an unsigned certificate of (D,D) convergence, then there exists an unsigned reduced certificate for (D,D) convergence with the same rate.
    Furthermore, for $\tilde\Lambda\in\R^{N\times N}$, $\tilde \Lambda\in\cP_\textup{cert}$ if and only if $(\tilde \Lambda)_V\in\cP_\textup{cert}$.
    The same holds for reduced (G,G) certificates, and reduced (G,D) certificates.
\end{lemma}

In the (D,G)--(D,G) setting of \cref{subsec:DG_duality}, we used the identity $\Lambda P \Phi = P^{-1}$ repeatedly. 
In the remaining settings, this quantity will need to be replaced with the following compressed versions. Its proof is deferred to \cref{app:deferred}.
\begin{lemma}
\label{lem:compressed_Lambda_P_Phi}
Let $\Lambda\in\R^{N\times N}$ be invertible and define $\Phi = (\bar P\Lambda \bar P)^{-1}$. Then,
\begin{gather*}
    \Lambda_U P \Phi_V = (I - e_0\bbm1^\intercal)P^{-1}\\
    \Lambda_W \bar P_W\Phi_W = \bar P^{-1}_W = (e_0\bbm1^\intercal - I + E_N)P^{-1},
\end{gather*}
where $E_N\coloneqq e_Ne_N^\intercal$.
\end{lemma}

Finally, the following lemma states that it suffices to prove restricted congruences.
\begin{lemma}
    \label{lem:restricted_congruence}
Let $\cL_1,\,\cL_2$ be symmetric matrices and suppose $\cL_1=\cR \cL_2\cR^\intercal$ and 
\begin{align*}
    \ker(\cR)\subseteq \ker(\cL_2).
\end{align*}
Then, $\cL_1\succeq 0$ if and only if $\cL_2\succeq 0$.
\end{lemma}
\begin{proof}
It is clear that $\cL_2\succeq 0$ implies $\cL_1\succeq 0$. Now, suppose $\cL_1\succeq 0$ and let $x$ be arbitrary.
We can decompose $x = y + z$ where $z\in \ker(\cL_2)$ and $y \in \range(\cL_2)$. By assumption, 
$\ker(\cR)\subseteq\ker(\cL_2)$ so that
$\range(\cL_2)\subseteq\range(\cR^\intercal)$. Thus, there exists $y'$ so that $\cR^\intercal y' = y$. Now,
\begin{equation*}
    x^\intercal \cL_2 x = (y + z)^\intercal \cL_2 (y+ z) = y^\intercal \cL_2 y = (y')^\intercal \cR \cL_2\cR^\intercal y' = (y')^\intercal \cL_1 y' \geq 0.\qedhere
\end{equation*}
\end{proof}

\subsection{Distance--Distance and Gradient--Gradient duality}
\label{subsec:DD_duality}

The following lemma replaces \cref{lem:dg_block_calculations}. Its proof is deferred to \cref{app:deferred_block_identities}.
\begin{lemma}
    \label{lem:DD_duality_blocks}
    Suppose $q>0$.
    Let $\tilde H$ strictly lower triangular. Define $A,b, w, \zeta$ as in \cref{lem:certificateLift} with respect to $\tilde H, (M_G, M_G)$. Define $A^\vee, b^\vee, w^\vee, \zeta^\vee$ analogously as in \cref{lem:certificateLift} with respect to $\tilde H^A, (M_D, M_D)$. Set $\eta = e_N^\intercal v_H = e_N^\intercal v_{H^A}$. Then, 
\begin{align*}
    &P^{-1}A^\vee P = A^\intercal - \frac{1}{q}e_0 v_H^\intercal
    && P w^\vee = b- \frac{1}{q}\bbm1
    && P^{-1} b^\vee = w + \eta e_0
    && \zeta^\vee = \zeta.
\end{align*}
Moreover,
\begin{align*}
    A\bbm1 = \frac{1}{q}\bbm1 \qquad\text{and}\qquad \bbm1^\intercal w = 0.
\end{align*}
\end{lemma}

\begin{theorem}
    \label{thm:DD}
Suppose $q>0$.
Let $\tilde H\in\R^{\cI_N\times \cI_N}$ be strictly lower triangular and let $\tau>0$. Let $\Lambda\in\R^{N\times N}$ be invertible and define
\begin{align*}
    \Phi \coloneqq (\bar P\Lambda \bar P)^{-1}\qquad\text{and}\qquad
    \cR \coloneqq \begin{bmatrix}
    \Lambda_U P\\
    & 0 & 1\\
     & 1 & 0
    \end{bmatrix}
\end{align*}
Then,
\begin{align*}
    \cL_{\Lambda_U,\tilde H}^\textup{G,G} = \cR \cL_{\Phi_V,\tilde H^A}^\textup{D,D}\cR^\intercal
\end{align*}
and $\Lambda_U$ is an unsigned certificate for (G,G) convergence of $\tilde H$ with rate $\tau$ if and only if $\Phi_V$ is an unsigned certificate for (D,D) convergence of $\tilde H^A$ with rate $\tau$.
Moreover, if $\Lambda$ is symmetric then $\Phi$ is also symmetric.
\end{theorem}

\begin{proof}
Throughout this proof, abbreviate $\cL_1 = \cL_{\Lambda_U,\tilde H}^\textup{G,G}$ and $\cL_2 = \cL_{\Phi_V,\tilde H^A}^\textup{D,D}$ and set $\eta = e_N^\intercal v_H = e_N^\intercal v_{H^A}$.



The (1,1) block of $\cR\cL_2\cR^\intercal$ is
\begin{align*}
    -2 \Lambda_{U}P\sym(\Phi_{V} A^\vee )P\Lambda_{U}^\intercal 
    & = -2 \sym(\Lambda_{U}P\Phi_{V}A^\vee P\Lambda_{U}^\intercal)
   \\
    & = -2 \sym((I - e_0\bbm1^\intercal)(A^\intercal - \frac{1}{q}e_0v_H^\intercal ) \Lambda_{U}^\intercal)\\
    & = -2 \sym(A^\intercal \Lambda_{U}^\intercal).
\end{align*}
Here, the second line uses \cref{lem:compressed_Lambda_P_Phi,lem:DD_duality_blocks} and
the third line uses $A\bbm 1 = \frac{1}{q}\bbm 1$, $\bbm1\in\ker(\Lambda_U)$, and $e_0 \in\ker(I - e_0\bbm1^\intercal)$.

The (1,2) and (1,3) blocks of $\cR\cL_2\cR^\intercal$ are
\begin{align*}
    \Lambda_U P w^\vee &= \Lambda_U (b - \frac{1}{q}\bbm1)=\Lambda_U b,\\
    \Lambda_U P \Phi_V b^\vee &= (I- e_0\bbm1^\intercal)P^{-1} b^\vee= (I- e_0\bbm1^\intercal)(w+\eta e_0)
    = w.
\end{align*}
Here, the first line uses \cref{lem:DD_duality_blocks} and $\bbm 1 \in\ker(\Lambda_U)$ and the second line uses \cref{lem:compressed_Lambda_P_Phi,lem:DD_duality_blocks}, $e_0 \in \ker(I - e_0\bbm1^\intercal)$, and $\bbm1^\intercal w = 0$.

The (2,3) entry of $\cR\cL_2\cR^\intercal$ is $\zeta^\vee = \zeta$ by \cref{lem:DD_duality_blocks}.

We conclude that $\cL_1 = \cR\cL_2 \cR^\intercal$. As $\bar P$ is symmetric, $\Lambda$ being symmetric implies that $\Phi$ is symmetric.

It remains to invoke \cref{lem:restricted_congruence}. The kernel of $\cR$ is generated by $[e_N , 0, 0]$.
A direct calculation shows that this vector lies in $\ker(\cL_2)$.
\end{proof}

H-duality follows as an immediate corollary whenever the sign conditions are preserved.
\begin{corollary}
    \label{cor:DD_convergence}
Fix $q>0$.
Let $\tilde H$ be strictly lower triangular and let $\tau >0$. Let $\Lambda_U$ be a PEP certificate for (G,G) convergence of $\tilde H$ with rate $\tau$ and suppose $\Lambda$ is invertible. If $\Phi\coloneqq (\bar P \Lambda \bar P)^{-1}$ satisfies the linear conditions $\Phi\in\cP_\textup{cert}$, then $\Phi_V$ is a PEP certificate for (D,D) convergence of $\tilde H^A$ with rate $\tau$.
Conversely, let $\Phi_V$ be a PEP certificate for (D,D) convergence of $\tilde H^A$ with rate $\tau$ and suppose $\Phi$ is invertible. If $\Lambda\coloneqq (\bar P\Phi\bar P)^{-1}$ satisfies the linear conditions $\Lambda\in\cP_\textup{cert}$, then $\Lambda_U$ is a PEP certificate for (G,G) convergence of $\tilde H$ with rate $\tau$.
\end{corollary}

\subsection{Gradient--Distance self-duality}
\label{subsec:GD_duality}

The following lemma replaces \cref{lem:dg_block_calculations}. Its proof is deferred to \cref{app:deferred_block_identities}.

\begin{lemma}
    \label{lem:gd_blocks}
    Suppose $q>0$. Let $\tilde H$ strictly lower triangular. Define $A, b, w, \zeta$ with respect to $\tilde H, (M_G, M_D)$ as in \cref{lem:certificateLift}. Define $A^\vee,b^\vee, w^\vee, \zeta^\vee$ analogously as in \cref{lem:certificateLift} with respect to $\tilde H^A, (M_G, M_D)$. Set $\eta = e_N^\intercal v_H = e_N^\intercal v_{H^A}$. Then,
\begin{align*}
    &\bar P^{-1}_W A^\vee \bar P_W = A^\intercal (I - E_N)
    && -\bar P_W w^\vee = b - \frac{1+q}{q}\eta e_N
    && -\bar P_W^{-1} b^\vee = w
    && \zeta^\vee = \zeta
\end{align*}
Moreover,
\begin{align*}
    e_N^\intercal w = e_N^\intercal w^\vee = 0.
\end{align*}
\end{lemma}

\begin{theorem}
    \label{thm:GD}
Suppose $q>0$.
Let $\tilde H\in\R^{\cI_N\times \cI_N}$ be strictly lower triangular and let $\tau>0$. Let $\Lambda\in\R^{N\times N}$ be invertible and define
\begin{align*}
    \Phi \coloneqq (\bar P\Lambda \bar P)^{-1}\qquad\text{and}\qquad
    \cR \coloneqq \begin{bmatrix}
    -\Lambda_W \bar P_W\\
    & 0 & 1\\
     & 1 & 0
    \end{bmatrix}
\end{align*}
Then,
\begin{align*}
    \cL_{\Lambda_W,\tilde H}^\textup{G,D} = \cR \cL_{\Phi_W,\tilde H^A}^\textup{G,D}\cR^\intercal
\end{align*}
and $\Lambda_W$ is an unsigned certificate for (G,D) convergence of $\tilde H$ with rate $\tau$ if and only if $\Phi_W$ is an unsigned certificate for (G,D) convergence of $\tilde H^A$ with rate $\tau$.
Moreover, if $\Lambda$ is symmetric then $\Phi$ is also symmetric.
\end{theorem}

\begin{proof}
Throughout this proof, abbreviate $\cL_1 = \cL_{\Lambda_W,\tilde H}^\textup{G,D}$ and $\cL_2 = \cL_{\Phi_W,\tilde H^A}^\textup{G,D}$ and set $\eta = e_N^\intercal v_H = e_N^\intercal v_{H^A}$.
Below, we verify that $\cL_1 = \cR \cL_2\cR^\intercal$.

The (1,1) block of $\cR \cL_2 \cR^\intercal$ is
\begin{align*}
    -2\Lambda_W \bar P_W\sym\left(\Phi_W A^\vee\right)\bar P_W\Lambda_W^\intercal 
    &= -2\sym\left(\Lambda_W\bar P_W\Phi_W A^\vee \bar P_W\Lambda_W^\intercal\right) \\
    &= -2\sym\left(\bar P_W^{-1} A^\vee \bar P_W\Lambda_W^\intercal\right) \\
    &= -2\sym\left(A^\intercal (I - E_N)\Lambda_W^\intercal\right) \\
    &= -2\sym\left(A^\intercal \Lambda_W^\intercal\right).
\end{align*}
Here, the second line uses \cref{lem:compressed_Lambda_P_Phi}, the third line uses \cref{lem:gd_blocks}, and the last line uses $e_N\in\ker(\Lambda_W)$.

The (1,2) and (1,3) blocks of $\cR\cL_2\cR^\intercal$ are
\begin{align*}
    -\Lambda_W\bar P_Ww^\vee
    &= \Lambda_W \left(b - \frac{1+q}{q}\eta e_N\right) = \Lambda_W b,\\
    -\Lambda_W \bar P_W \Phi_W b^\vee &= -\bar P_W^{-1} b^\vee = w.
\end{align*}
Here, the first line uses \cref{lem:gd_blocks} and
$e_N \in\ker(\Lambda_W)$. The second line uses 
\cref{lem:compressed_Lambda_P_Phi,lem:gd_blocks}.

The (2,3) entry of $\cR\cL_2\cR^\intercal$ is $\zeta^\vee = \zeta$ by \cref{lem:gd_blocks}.

We conclude that $\cL_1 = \cR\cL_2\cR^\intercal$. As $\bar P$ is symmetric, $\Lambda$ being symmetric implies $\Phi$ is symmetric.

It remains to invoke \cref{lem:restricted_congruence}. The kernel of $\cR$ is generated by $[e_N , 0, 0]$.
A direct calculation shows that this vector lies in $\ker(\cL_2)$.
\end{proof}

\begin{corollary}
    \label{cor:GD_convergence}
    Fix $q>0$.
Let $\tilde H$ be strictly lower triangular and let $\tau >0$. Let $\Lambda_W$ be a PEP certificate for (G,D) convergence of $\tilde H$ with rate $\tau$ and suppose $\Lambda$ is invertible. If $\Phi\coloneqq (\bar P \Lambda \bar P)^{-1}$ satisfies the linear conditions $\Phi\in\cP_\textup{cert}$, then $\Phi_W$ is a PEP certificate for (G,D) convergence of $\tilde H^A$ with rate $\tau$.
\end{corollary}

\subsection{Tridiagonal certificates}
\label{sec:tridiagonal}
We will now give one special case in which the transformation sending $\Lambda$ to $(P\Lambda P)^{-1}$ not only sends unsigned certificates to unsigned certificates, but furthermore transforms PEP certificates into PEP certificates. Specifically, we will consider the situation in which $\Lambda$ has a tridiagonal sparsity pattern. Recall that $\Lambda$ is said to be \emph{tridiagonal} if it satisfies $\Lambda_{ij} = 0$ whenever $|i - j| > 1$.

Tridiagonal certificates often arise when the underlying proof of the convergence rate can be viewed in terms of the monotonicity of a Lyapunov function which evolves from iteration to iteration \cite{Kim2024-Hduality}.
Tridiagonal certificates are also of importance because many optimal algorithms are known to have tridiagonal certificates, including OGM, ITEM, Picard iteration, and the optimal Halpern iteration \cite{kim2017convergence, Drori2021OnTO, Lieder2020OnTC}.

The following proposition is the main result of this subsection and states that if an algorithm has a tridiagonal PEP certificate in one of the performance setups in \eqref{eq:classical_dualities}, then its H-dual automatically has the same guarantee on the dual performance setup.

\begin{proposition}
    \label{prop:tridiag_dual}
    Suppose that $\Lambda$ is an invertible reduced PEP certificate for
    $\tilde H$ with rate $\tau>0$ and one of the performance setups (D,D), (D,G), (G,D) or (G,G). If $\Lambda$ is further tridiagonal, then $\Phi = (P\Lambda P)^{-1}$ is a reduced PEP certificate for $\tilde H^A$ with rate $\tau$ for the performance setup (G,G), (D,G), (G,D) or (D,D) respectively.
\end{proposition}

Because we have already established that the transformation sending $\Lambda$ to $(P\Lambda P)^{-1}$ sends unsigned reduced certificates to unsigned reduced certificates,
it suffices to check that $(P\Lambda P)^{-1}$ sends invertible tridiagonal elements of $\cP_\textup{cert}$ to $\cP_\textup{cert}$.

Before proving \cref{prop:tridiag_dual}, we will note the following linear algebraic facts, which are proved in \cite[Chapter 6]{berman1994nonnegative}, \cite{carlson1979schur}, and \cite{nabben1999decay}.
Recall that an $M$-matrix is a matrix whose off-diagonal entries are all nonpositive and whose eigenvalues have nonnegative real part.

\begin{lemma}
    \label{lem:mmatrix}
    Suppose that $\Lambda \in \cP_\textup{cert}$ is invertible. Then
    \begin{itemize}
        \item $-\Lambda$ is an $M$-matrix and $\Lambda^{-1}$ has only nonpositive entries
        \item Any Schur complement of $\Lambda$ is also in $\cP_\textup{cert}$
        \item 
        Suppose further that $\Lambda$ is tridiagonal and set $v = \Lambda^{-1} e_0$. Then, $v_0 \le v_1 \le \dots \le v_N \le 0$.
    \end{itemize}
\end{lemma}
%
%

\begin{proof}[Proof of \cref{prop:tridiag_dual}]
    Let $\Lambda\in \cP_\textup{cert}$ be invertible and tridiagonal. We will show that $(P\Lambda P)^{-1}\in\cP_\textup{cert}$. Since $\cP_\textup{cert}$ is invariant under simultaneously reordering rows and columns, 
    we have that $\Gamma = R\Lambda R$ is also an invertible tridiagonal element of $\cP_\textup{cert}$. We write
    \[
        (P\Lambda P)^{-1} = (S\Gamma S^\intercal)^{-1} = S^{-\intercal} \Gamma^{-1}S^{-1}.
    \]
    Thus, our goal is to establish:
    \begin{align*}
        S^{-\intercal}\Gamma^{-1} S^{-1} \bo \le 0\\
        (S^{-\intercal}\Gamma^{-1} S^{-1})^{\intercal} \bo \le 0 \\
        \forall i \neq j,\;e_j^{\intercal}(S^{-\intercal}\Gamma^{-1} S^{-1})e_i \ge 0
    \end{align*}

    We will first show that $S^{-\intercal}\Gamma^{-1} S^{-1} \bo \le 0$. First, note that 
    \[
        S^{-\intercal}\Gamma^{-1} S^{-1} \bo = S^{-\intercal}\Gamma^{-1} e_0.
    \]

    Letting $v = \Gamma^{-1}e_0$, we see that
    \[
        S^{-\intercal}\Gamma^{-1} S^{-1} \bo = S^{-\intercal} v = \begin{pmatrix}v_0 - v_1 \\ v_1 - v_2 \\ \dots \\ v_{N-1} - v_N \\v_N\end{pmatrix} \le 0.
    \]
    The final inequality follows from the third item of \cref{lem:mmatrix}.
    The fact that $(S^{-\intercal}\Gamma^{-1} S^{-1})^{\intercal} \bo \le 0$ follows similarly.

    Thus, it remains to show that the off-diagonal entries of $S^{-\intercal}\Gamma^{-1} S^{-1}$ are nonnegative. For this, the essential idea is to again apply the third item of \cref{lem:mmatrix}, but to a slightly modified matrix that depends on the column index $i$.

    We will first show that $(S^{-\intercal}\Gamma^{-1} S^{-1})_{ji} \ge 0$ in the case where $i = N$ and $j < N$. Consider
    \[
        S^{-\intercal}\Gamma^{-1} S^{-1} e_N = S^{-\intercal}\Gamma^{-1} e_N.
    \]
    Let $v=\Gamma^{-1} e_N = R\Lambda^{-1}e_0$.
    Then, by the third item of \cref{lem:mmatrix}, $v_N \le v_{N-1} \le \dots \le v_0 \le 0$.
    Thus, $S^{-\intercal}\Gamma^{-1} S^{-1} e_N = \begin{pmatrix}v_0 - v_{1} & \dots & v_{N-1}-v_N & v_N \end{pmatrix}$, and so for all $j < N$, $e_j^{\intercal} \Gamma^{-1} e_N \ge 0$.

    Now, fix some $i \in \cI_N \setminus N$ and consider
    \[
        S^{-\intercal}\Gamma^{-1} S^{-1} e_i = S^{-\intercal}\Gamma^{-1} (e_i - e_{i+1}).
    \]
    Let $v = \Gamma^{-1} (e_i - e_{i+1})$ and consider the following block form of the equation defining $v$.
    \[
        \begin{pmatrix}
            \Gamma_L & \Gamma_{i,i+1} e_i e_0^{\intercal}\\
            \Gamma_{i+1,i} e_0 e_i^{\intercal} & \Gamma_R
        \end{pmatrix}
        \begin{pmatrix} v_L \\ v_R\end{pmatrix} = 
        \begin{pmatrix} e_i \\ -e_0\end{pmatrix}.
    \]
    This gives two block equations
    \[
        \Gamma_L v_L + \Gamma_{i,i+1}v_{i+1}e_i = e_i,
    \]
    \[
        \Gamma_R v_R + \Gamma_{i+1,i}v_{i}e_0 = -e_0.
    \]
    We may massage these equations slightly to become
    \[
        R\Gamma_L R (Rv_L) = c_L e_0,
    \]
    \[
        \Gamma_R v_R = c_R e_0,
    \]
    where $c_L = 1-\Gamma_{i\;i+1}v_{i+1}$ and $c_R = -1-\Gamma_{i+1\;i}v_{i}$.
    It is also clear that $(R\Gamma_LR),\Gamma_R\in\cP_\textup{cert}$.
    Thus, as long as $c_L \ge 0$, and  $c_R \le 0$, then the third item of \cref{lem:mmatrix} implies that $v_0 \ge v_1 \ge \dots \ge v_i$, and $v_{i+1} \ge v_{i+2} \ge \dots \ge v_N$. This would then imply that for all $j \neq i$, $e_j^{\intercal}(S^{-\intercal}\Gamma^{-1} S^{-1})e_i = v_{j} - v_{j+1} \ge 0$.
    Therefore, it remains to show that $c_L \ge 0$ and $c_R \le 0$.

    For this, we first note that 
    \[
        \Gamma_L v_L = c_L e_i,
    \]
    so $v_i = c_L (\Gamma_L^{-1})_{ii}$, and since $\Gamma_L^{-1}$ has only nonpositive entries, this implies that $c_L$ is nonnegative if $v_i$ is nonpositive. Similarly, we see that $c_R$ is nonpositive if $v_{i+1}$ is nonnegative. Thus, it suffices to show that $v_{i+1} \ge 0 \ge v_i$.

    For this, we will let $S = \{i,i+1\}$ and consider the block decomposition
    \[
        \begin{pmatrix}
            \Gamma_{SS} & \Gamma_{S S^c}\\
            \Gamma_{S^cS} & \Gamma_{S^c S^c}
            \end{pmatrix} \begin{pmatrix}v_S \\ v_{S^c} \end{pmatrix} = \begin{pmatrix}\begin{pmatrix}1\\-1\end{pmatrix} \\ 0\end{pmatrix}.
    \]
    From this, we can arrive at the equation that 
    \[
        v_S = C^{-1} \begin{pmatrix}1 \\ -1\end{pmatrix},
    \]
    where $C$ is the Schur complement of $\Gamma$ with respect to $S^c$, i.e. 
    \[
        C = \Gamma_{SS} - \Gamma_{SS^c}\Gamma_{S^cS^c}^{-1}\Gamma_{S^cS}.
    \]

    Parts 1 and 2 of \cref{lem:mmatrix} imply that $C$ has nonnegative off-diagonal entries and nonpositive column sums.
    If we write
    \[
        C = \begin{pmatrix} a & b \\ c & d\end{pmatrix},
    \]
    then
    \[
        C^{-1} = \frac{1}{\det(C)}\begin{pmatrix} d & -b \\ -c & a\end{pmatrix}.
    \]
    Hence, 
    \[
        C^{-1}\begin{pmatrix}1\\-1\end{pmatrix} = \frac{1}{\det(C)}\begin{pmatrix} d+b \\ -c - a\end{pmatrix}.
    \]
    Since $\det(C) > 0$ and also the column sums of $C$ are nonpositive, it follows that $v_i \le 0$ and $v_{i+1} \ge 0$.
\end{proof}

\subsubsection{ITEM and ITEM-GG}
\label{subsubsec:ITEM_GG}
The ITEM algorithm is defined by the following recurrence: let $z_0 = y_0$ and then set
\begin{align*}
    y_k &= (1-\beta_k)z_k + \beta_k x_k\\
    z_{k+1} &= (1-\mu \delta_k)z_k + \mu \delta_k y_k - \frac{\delta_k}{L}s_k,
\end{align*}
where $x_k$ is by definition $y_k - \nabla f(y_k)$. Here, $\beta_k$ and $\delta_k$ are defined in terms of the sequence $A_0=0,A_1,\dots$ using the mutual recurrence
\[
    A_{k+1} = (1+\mu)A_k+2(1+\sqrt{(1+A_k)(1+\mu A_k)}),
\]
\[
    \beta_{k} = \frac{A_k}{(1+\mu)A_{k+1}},
\]
\[
    \delta_{k} = \sqrt{\frac{A_{k+1}}{1+\mu A_{k+1}}}.
\]
Note that $A_k>0$ for all $k\geq 1$.

\cite{taylor2022optimal} proves the following guarantee for ITEM
\[
    \|z_N\|^2 \le \frac{1}{1+\mu A_N}\|y_0\|^2
\]
using an explicit PEP certificate.
Rescaling their certificate using our convention gives the following reduced PEP certificate of (D,D) convergence for ITEM:
\[
    \Lambda_{ij} = \frac{1-\mu}{2\sqrt{1+\mu A_N}}\begin{cases}
                        -A_{i+1} \text{ if }i = j\\
                        A_i \text{ if } i = j+1\\
                        0 \text{ otherwise}
    \end{cases}.
\]
Note that this certificate is lower bidiagonal (hence tridiagonal) with negative entries on the diagonal and thus is invertible.
\cref{prop:tridiag_dual} implies that $\Phi = (P\Lambda P)^{-1}$ is a reduced (G,G) PEP certificate for the H-dual of ITEM. In particular, the H-dual of ITEM produces an output which satisfies
\[
    \|\nabla h(y_N)\|^2 \le \frac{1}{1+\mu A_N}\|\nabla h(y_0)\|^2.
\]

\section{Applications}
\label{sec:examples}

This section collects applications of \cref{thm:DD,thm:DG,thm:GD}.
We apply our transformation $\Lambda \leftrightarrow \Phi$ in each of the (D,D)--(G,G), (D,G)--(D,G), (G,D)--(G,D) settings to FSFOMs that are either minimax optimal or conjectured to be minimax optimal.
In each case, we record the $\Phi$ produced by starting from a PEP certificate $\Lambda$ and applying our map $\Lambda\mapsto\Phi$ from \cref{thm:DD,thm:DG,thm:GD}.
Recall that \cref{thm:DD,thm:DG,thm:GD} only guarantee that $\Phi$ will be an unsigned certificate. Nonetheless, in all examples below, the returned $\Phi$ satisfies the requisite sign conditions and is a valid PEP certificate.

In the settings where a minimax optimal FSFOM is not known, we use a numerically optimized choice of $\tilde H$ and $\Lambda$. This $\tilde H$ and $\Lambda$ pair are found by repeatedly solving a local first-order approximation to $\min_{\tilde H}\tau_\textup{PEP}^\textup{*,*}(\tilde H)$ in the $\tilde H, \Lambda, \tau$ variables simultaneously. Implementation details can be found here:

\begin{center}
\url{https://github.com/alexlihengwang/H-Duality}
\end{center}

\begin{table}[h]
\centering
\begin{tabular}{llll}
\toprule
Setting  & Duality & Primal Method & Example\\
\midrule
Function 
    & (D,D)--(G,G)& ITEM & \cref{ex:item}\\
Function 
    & (D,G)--(D,G)& optimized method
    & \cref{ex:function_dg}\\
Function 
    & (G,D)--(G,D)& optimized method
    & \cref{ex:function_gd}\\
Operator 
    & (D,D)--(G,G)& Picard
    & \cref{ex:op_picard}\\
Operator 
    & (D,G)--(D,G)& OHM
    & \cref{ex:op_OHM}\\
Operator 
    & (G,D)--(G,D)& optimized method
    & \cref{ex:op_gd}\\
\bottomrule
\end{tabular}
\end{table}

\begin{example}
    \label{ex:item}
The Information-Theoretic Exact Method (ITEM) due to~\citet{drori2017exact,taylor2023optimal} is a minimax optimal algorithm for the (D,D) performance criteria in the function setting. For $N = 5$ and $q = 0.1$, the $\tilde H$ matrix for ITEM has the following form
\begin{align*}
    \tilde H = \begin{bmatrix}
 0.0  &   0.0   &  0.0   &  0.0   &  0.0  &   0.0\\
 1.553  & 0.0   &  0.0   &  0.0    & 0.0  &   0.0\\
 0.1188 & 1.8534 & 0.0   &  0.0    & 0.0  &   0.0\\
 0.0358 & 0.2569 & 1.9723 & 0.0    & 0.0  &   0.0\\
 0.0122 & 0.0873 & 0.3304 & 2.0254 & 0.0  &   0.0\\
 0.0092 & 0.0661 & 0.2501 & 0.7763 & 3.2214 & 0.0
    \end{bmatrix}.
\end{align*}
\citet{taylor2023optimal} provide an analytic PEP certificate $\Lambda_V$ for ITEM in this setting. For $N = 5$ and $q = 0.1$, this is numerically
\begin{align*}
    \Lambda &= \begin{bmatrix}
    -0.5235 &  0.0   &    0.0     &  0.0   &     0.0\\
  0.5235 &  -1.7118  &  0.0   &    0.0    &    0.0\\
  0.0    &   1.7118 &  -4.1981 &   0.0    &    0.0 \\
  0.0   &    0.0   &    4.1981 &  -9.3178  &   0.0 \\
  0.0   &    0.0   &    0.0   &    9.3178 &  -19.8229
    \end{bmatrix}.
\end{align*}
Applying \cref{thm:DD} produces
\begin{align*}
    \Phi &= (\bar P\Lambda \bar P)^{-1} = \begin{bmatrix}
-0.1073 & 0.0  &    0.0    &  0.0   &   0.0569\\
  0.1073 & -0.2382  & 0.0    &  0.0   &   0.1309\\
  0.0    &  0.2382&  -0.5842 &  0.0   &   0.346\\
  0.0   &   0.0    &  0.5842 & -1.9102  & 1.3261\\
  0.0   &   0.0   &   0.0    &  1.9102 & -1.9102
    \end{bmatrix}.
\end{align*}
The sign conditions on $\Phi$ and its row and column sums are satisfied so $\Phi_U$ is a PEP certificate of (G,G) convergence of $\tilde H^A$ with the same rate. We refer to the algorithm $\tilde H^A$ as ITEM-GG.
The convergence rate of ITEM-GG can be proved analytically in this way; see \cref{subsubsec:ITEM_GG}.
\end{example}

\begin{example}
    \label{ex:function_dg}
Let $N = 5$ and $q = 0.1$. 
The following optimized FSFOM $\tilde H$ and its PEP certificate $\Lambda$ for the (D,G) performance criteria in the function setting were found by repeatedly solving a local first-order approximation to $\min_{\tilde H}\tau_\textup{PEP}^\textup{D,G}(\tilde H)$:
\begin{align*}
    \tilde H &= \begin{bmatrix}
0.0  &   0.0  &   0.0   &  0.0    & 0.0 &  0.0\\
 1.55  &  0.0  &   0.0  &   0.0   &  0.0 &  0.0\\
 0.1137 & 1.8318 & 0.0  &   0.0   &  0.0  & 0.0\\
 0.0304 & 0.2227 & 1.8998 & 0.0   &  0.0  & 0.0\\
 0.0075 & 0.0551 & 0.2227 & 1.8318 & 0.0  & 0.0\\
 0.001  & 0.0075 & 0.0304 & 0.1137 & 1.55 & 0.0
    \end{bmatrix}\\
    \Lambda &= \begin{bmatrix}
-0.1158 &  0.0    &  0.0 &  0.0    &  0.0   &   0.0161\\
  0.1158 & -0.3852  & 0.0  & 0.0   &   0.0   &   0.1126\\
  0.0    &  0.3852 & -1.0 &  0.0    &  0.0   &   0.418\\
  0.0    &  0.0    &  1.0 & -2.5963&   0.0   &   1.3837\\
  0.0    &  0.0   &   0.0 &  2.5963 & -8.6383  & 5.847\\
  0.0    &  0.0  &    0.0  & 0.0    &  8.6383  &-8.7768
    \end{bmatrix}.
\end{align*}
Note that $\tilde H$ is self-antitranspose.
Additionally, $\Lambda = (P\Lambda P)^{-1}$ is a fixed point of the map $\Lambda\mapsto\Phi$ of \cref{thm:DG}. These observations suggest a promising path to analytically describing ``ITEM-DG''. Concurrent work by \citet{kim2026domain} gives an analytic description of this algorithm in the non-strongly convex setting and links it to the ``lemniscate constant''.
\end{example}

\begin{example}
    \label{ex:function_gd}
Let $N = 5$ and $q = 0.1$. 
The following optimized FSFOM for the (G,D) performance criteria in the function setting was found by repeatedly solving a local first-order approximation to $\min_{\tilde H}\tau_\textup{PEP}^\textup{G,D}(\tilde H)$.
\begin{align*}
    \tilde H = \begin{bmatrix}
 0.0   &  0.0  &   0.0  &   0.0   &  0.0   &  0.0\\
 3.3166 & 0.0   &  0.0  &   0.0   &  0.0   &  0.0\\
 0.8684 & 2.0733&  0.0  &   0.0   &  0.0   &  0.0\\
 0.3255 & 0.4024 & 2.0734 & 0.0   &  0.0   &  0.0\\
 0.122  & 0.1508 & 0.4024 & 2.0734 & 0.0   &  0.0\\
 0.0987 & 0.122  & 0.3255 & 0.8684 & 3.3166 & 0.0
    \end{bmatrix}.
\end{align*}
The PEP certificate found in this way is $\Lambda_{W}$ where
\begin{align*}
    \Lambda &= \begin{bmatrix}
    -4.2011 &  2.0497 &  0.0 &  0.0  &    0.0\\
  0.0    & -2.0497  & 1.0  & 0.0     & 0.0\\
  0.0    &  0.0    & -1.0  & 0.4879  & 0.0\\
  0.0    &  0.0    &  0.0 & -0.4879  & 0.238\\
  0.0    &  0.0    &  0.0 &  0.0    & -0.238\\
    \end{bmatrix}.
\end{align*}
As in the previous example, $\tilde H$ is self-antitranspose and $\Lambda = (\bar P\Lambda \bar P)^{-1}$ turns out to be a fixed point of the map $\Lambda\mapsto\Phi$ of \cref{thm:GD}. These observations suggest a promising path to analytically describing ``ITEM-GD''.
\end{example}

\begin{example}
    \label{ex:op_OHM}
The Optimized Halpern Method (OHM)~\cite{Lieder2020OnTC,Kim2021AccelPPM,park2022exact} is a minimax optimal algorithm for the (D,G) performance criteria in the operator setting. For $N = 5$ and $q = 0.1$, the $\tilde H$ matrix for OHM has the following form
\begin{align*}
    \tilde H = \begin{bmatrix}
0.0  &    0.0 &     0.0   &   0.0   &  0.0  &   0.0\\
  1.082  &  0.0 &     0.0  &    0.0  &   0.0  &   0.0\\
 -0.2397 &  1.4272&   0.0   &   0.0  &   0.0  &   0.0\\
 -0.1123 & -0.1903 &  1.5889 &  0.0  &   0.0  &   0.0\\
 -0.0619 & -0.1048 & -0.1347 &  1.678 &  0.0  &   0.0\\
 -0.0371 & -0.0629 & -0.0808 & -0.0933 & 1.7314 & 0.0
    \end{bmatrix}.
\end{align*}
\citet{park2022exact} provide an analytic PEP certificate for OHM in this setting. For $N = 5$ and $q = 0.1$, this is numerically
\begin{align*}
    \Lambda &= \begin{bmatrix}
-0.2812 &  0.2812  & 0.0   &   0.0   &   0.0   &   0.0\\
  0.2812 & -1.1627 &  0.8815  & 0.0   &   0.0   &   0.0\\
  0.0    &  0.8815 & -2.763   & 1.8815 &  0.0   &   0.0\\
  0.0    &  0.0    &  1.8815 & -5.2973 &  3.4157 &  0.0\\
  0.0    &  0.0    &  0.0    &  3.4157 & -9.1061 &  5.6904\\
  0.0    &  0.0    &  0.0    &  0.0    &  5.6904 & -6.6905
    \end{bmatrix}.
\end{align*}
Applying \cref{thm:DG} produces
\begin{align*}
    \Phi &= (P\Lambda P)^{-1} = \begin{bmatrix}
    -0.1757 &  0.0   &   0.0   &   0.0    &  0.0  &    0.1757\\
  0.0   &  -0.2928&   0.0   &   0.0  &    0.0  &    0.2928\\
  0.0   &   0.0  &   -0.5315 & 0.0  &    0.0   &   0.5315\\
  0.0   &   0.0   &   0.0   &  -1.1344  & 0.0   &   1.1344\\
  0.0   &   0.0   &   0.0   &   0.0  &   -3.5558  & 3.5558\\
  0.1757 &  0.2928  & 0.5315  & 1.1344  & 3.5558 & -6.6901
    \end{bmatrix}.
\end{align*}
The sign conditions on $\Phi$ and its row and column sums are satisfied, so $\Phi$ is a PEP certificate of (D,G) convergence of $\tilde H^A\neq \tilde H$ with the same rate.
\end{example}

\begin{example}
    \label{ex:op_picard}
Picard iteration is the minimax optimal method for the (D,D) performance criteria in the operator setting. For $N = 5$ and $q = 0.1$, its $\tilde H$ matrix is given by
\begin{align*}
    \tilde H = \begin{bmatrix}
0.0  &   0.0   &  0.0   &  0.0   &  0.0 &    0.0\\
 1.8333 & 0.0  &   0.0  &   0.0  &   0.0  &   0.0\\
 0.0  &   1.8333 & 0.0   &  0.0  &   0.0   &  0.0\\
 0.0  &   0.0   &  1.8333 & 0.0  &   0.0   &  0.0\\
 0.0  &   0.0   &  0.0    & 1.8333 & 0.0   &  0.0\\
 0.0  &   0.0   &  0.0    & 0.0    & 1.8333 & 0.0
    \end{bmatrix}.
\end{align*}
A PEP certificate is given by $\Lambda_{V}$ where
\begin{align*}
    \Lambda = \begin{bmatrix}
    -0.8841 &  0.0   &   0.0  &    0.0  &  0.0\\
  0.0   &  -1.2731 &  0.0   &   0.0  &  0.0\\
  0.0   &   0.0  &   -1.8333  & 0.0  &  0.0\\
  0.0   &   0.0  &    0.0  &   -2.64 &  0.0\\
  0.0   &   0.0  &    0.0  &    0.0 &  -3.8016
    \end{bmatrix}.
\end{align*}
Note that this is a representation of the standard proof of the convergence rate of Picard iteration given by
\[
    \|y_0 - y_\star\| \ge \frac{1+\mu}{1-\mu} \|y_1 - y_\star\| \ge \dots \ge
    \left(\frac{1+\mu}{1-\mu}\right)^n \|y_n - y_\star\|,
\]
up to normalization.

Applying \cref{thm:DD} produces
\begin{align*}
    \Phi &= (\bar P\Lambda \bar P)^{-1} = \begin{bmatrix}
    -0.6418  & 0.3788   &0.0  &   0.0  &    0.0\\
  0.3788 & -0.9242 &  0.5455 &  0.0   &   0.0\\
  0.0   &   0.5455 & -1.3309  & 0.7855  & 0.0\\
 0.0   &   0.0    &  0.7855 & -1.9165 &  1.1311\\
  0.0   &   0.0   &   0.0    &  1.1311 & -1.1311
    \end{bmatrix}.
\end{align*}
The sign conditions on $\Phi$ and its row and column sums are satisfied, so $\Phi_{U}$ is a PEP certificate of (G,G) convergence of $\tilde H^A = \tilde H$ with the same rate.
This is a representation of the standard proof of the convergence rate of Picard iteration given by
\[
    \|y_0 - Ty_0\| \ge \frac{1+\mu}{1-\mu} \|y_1 - Ty_1\| \ge \dots \ge
    \left(\frac{1+\mu}{1-\mu}\right)^n \|y_n - T y_n\|,
\]
up to normalization.
\end{example}

\begin{example}
    \label{ex:op_gd}
Let $N = 5$ and $q = 0.1$. 
The following optimized FSFOM for the (G,D) performance criteria in the operator setting was found by repeatedly solving local first-order approximations to $\min_{\tilde H} \tau_\textup{PEP}^\textup{G,D}(\tilde H)$:
\begin{align*}
    \tilde H = \begin{bmatrix}
    0.0 & 0.0    & 0.0 &    0.0  &   0.0  &   0.0\\
 6.0 & 0.0  &   0.0  &   0.0    & 0.0   &  0.0\\
 0.0 & 1.8333 & 0.0  &   0.0 &    0.0  &   0.0\\
 0.0 & 0.0  &   1.8333 & 0.0  &   0.0  &   0.0\\
 0.0 & 0.0  &   0.0   &  1.8333 & 0.0  &   0.0\\
 0.0 & 0.0  &   0.0   &  0.0   &  1.8333 & 0.0
    \end{bmatrix}.
\end{align*}
The PEP certificate found in this way is $\Lambda_{W}$ where
\begin{align*}
    \Lambda &= \begin{bmatrix}
     -0.6336 &  0.0  &  0.0  &   0.0 &     0.0\\
  0.0  &   -0.44 &  0.0  &   0.0  &   0.0\\
  0.0  &    0.0  & -0.3056  & 0.0  &   0.0\\
 0.0  &   0.0 &   0.0 &   -0.2122 & 0.0\\
  0.0  &   0.0 &  0.0  &   0.0   &  -0.4823
    \end{bmatrix}.
\end{align*}
Applying \cref{thm:GD} produces
\begin{align*}
   \Phi &= (\bar P\Lambda \bar P)^{-1} = \begin{bmatrix}
    -6.7863 &  4.7127  & 0.0   &   0.0  &   0.0\\
  4.7127  &-7.9855  & 3.2727&   0.0   &   0.0\\
  0.0   &   3.2727 & -5.5455 &  2.2727 &  0.0\\
  0.0  &    0.0   &   2.2727 & -3.851  &  1.5783\\
 0.0  &    0.0   &   0.0   &   1.5783&  -1.5783
    \end{bmatrix}
\end{align*}
It is straightforward to verify that $\Phi$ has nonpositive row and column sums and nonnegative off-diagonal entries, so that $\Phi_{W}$ is a PEP certificate of (G,D) convergence of $\tilde H^A\neq \tilde H$ with the same rate.
\end{example}

\bibliographystyle{plainnat}
\bibliography{bib}

\appendix
\crefalias{section}{appendix}

\section{Deferred proofs}
\label{app:deferred}

\begin{proof}[Proof of \cref{lem:primal_dual_extremal}]
It suffices to show that
\begin{align}
    \label{eq:tau_extremal_primal_dual}
    \tau_\textup{unsigned}^{M_0,M_N}(\tilde H) &= \max_{\sigma\in\Omega}\abs{\kappa(\sigma, \tilde H)}\sqrt{\frac{d_\sigma^* M_N d_\sigma}{d_\sigma^* M_0 d_\sigma}}.
\end{align}

First, rewrite the SDP defining $\tau_\textup{unsigned}^{M_0,M_N}(\tilde H)$. Below, let
$t = \tau^2$ and $\Lambda_\tau = \tau\Lambda$. We will simplify the slack term using the relation $\bx = - SH\bg + v_H y_0$.
\begin{align*}
    \left(\tau_\textup{unsigned}^{M_0,M_N}(\tilde H)\right)^2&= 
    \inf_{\substack{\Lambda\in\R^{\cI_N\times \cI_N}\\\tau>0}}\set{\tau^2:\, 
    \frac{\tau}{2}Q_{M_0}(y_0, s_0) - \frac{1}{2\tau}Q_{M_N}(y_N,s_N) +
    \tr(\bg^\intercal\Lambda\bx)\geq 0\quad\forall (y_0,\bg)}\\
    &=
    \inf_{\substack{\Lambda_\tau\in\R^{\cI_N\times \cI_N}\\ t>0}}\set{t:\, 
    \frac{t}{2}Q_{M_0}(y_0, s_0) - \frac{1}{2}Q_{M_N}(y_N, s_N) + \tr(\bg^\intercal\Lambda_\tau \bx)\geq 0\quad\forall (y_0,\bg)}.
\end{align*}
The proof then splits into the function and operator settings.

In the function setting, $\Lambda\in\R^{\cI_N\times \cI_N}$ is a free variable. Thus, by strong duality,
\begin{align*}
    \left(\tau_\textup{unsigned}^{M_0,M_N}(\tilde H)\right)^2&= \sup_{y_0,\bg}\set{\frac{Q_{M_N}(y_N, s_N)}{Q_{M_0}(y_0, s_0)}:\, 
    \ip{g_i, x_j} = 0\qquad\forall i,j \in \cI_N
    },
\end{align*}
where $x_i, y_i, s_i, g_i$ are understood to be functions of $(y_0,\bg)$.

Fix an arbitrary feasible solution and let $V_1 = \spann\set{x_0,\dots, x_N}$ and $V_2 = \spann\set{g_0,\dots, g_N}$. The constraint
$\ip{g_i, x_j} = 0$ for all $i,j\in \cI_N$ shows that $V_1\perp V_2$. Now, recall
\begin{align*}
    s_i = qx_i + g_i, \qquad y_i = x_i + s_i = (1+q)x_i + g_i.
\end{align*}
Let $\Pi_1$ and $\Pi_2$ denote the projections onto $V_1$ and $V_2$ respectively.
Then,
\begin{align*}
    \Pi_1 s_i &= qx_i = \frac{q}{1+q}\Pi_1y_i = \mu \Pi_1 y_i\qquad\text{and}\qquad
    \Pi_2 s_i = g_i = \Pi_2 y_i.
\end{align*}
Note that on each subspace $V_\ell$, we have that $\Pi_\ell s_i = \sigma_\ell \Pi_\ell y_i$ for some $\sigma_\ell \in \Omega_\textup{fun}$.

The same conclusion holds in the operator setting. Since $\Lambda\in\S^{\cI_N}$ is a free variable, strong duality gives
\begin{align*}
    \left(\tau_\textup{unsigned}^{M_0,M_N}(\tilde H)\right)^2&= \sup_{y_0,\bg}\set{\frac{Q_{M_N}(y_N, s_N)}{Q_{M_0}(y_0, s_0)}:\, 
    \ip{g_i, x_j} + \ip{x_i , g_j} = 0\qquad\forall i,j\in \cI_N
    },
\end{align*}
where $x_i, y_i, s_i, g_i$ are understood to be functions of $(y_0,\bg)$. 
Let $z_i \coloneqq y_i - \frac{2}{1+\mu}s_i$ so that any interpolating operator satisfies $T(y_i) = z_i$.
Also define $x_\star = z_\star = y_\star = s_\star = g_\star = 0$.

Note that when $i=j$, the constraints require $\ip{x_i, g_i} = 0$. Thus,
\begin{align*}
    \ip{g_i - g_j, x_i - x_j} = -\ip{g_i, x_j} - \ip{g_j, x_i } = 0\qquad\forall i,j\in \cI_N\cup\set{\star}.
\end{align*}
Here, the identities involving $\star$ follow from the fact that $x_\star = g_\star = 0$.
Equivalently (see proof of \cref{lem:interpolation}), 
\begin{align*}
    \norm{z_i - z_j}^2 = \left(\frac{1-\mu}{1+\mu}\right)^2\norm{y_i - y_j}^2 ,\qquad \forall i,j\in \cI_N\cup\set{\star}.
\end{align*}
When $i\in \cI_N$ and $j= \star$, this requires $\norm{z_i}^2 = \left(\frac{1-\mu}{1+\mu}\right)^2\norm{y_i}^2$. In particular,
\begin{align*}
    \ip{z_i,z_j} = \left(\frac{1-\mu}{1+\mu}\right)^2\ip{y_i, y_j}\qquad\forall i,j\in \cI_N\cup\set{\star}.
\end{align*}
Since all inner products are preserved (up to a common scaling), there exists a unique linear isometry $U:\spann\set{y_i}\mapsto\spann\set{z_i}$ satisfying $U y_i = \frac{1+\mu}{1-\mu} z_i$. This isometry can be extended to $\R^d$, i.e., to an orthogonal matrix in $\R^{d\times d}$. Then,
\begin{align*}
    z_i = \frac{1-\mu}{1+\mu} U y_i.
\end{align*}
By the real canonical form for orthogonal operators, there is an orthogonal decomposition
\begin{align*}
    \R^d = \bigoplus_{\ell=1}^k V_\ell,
\end{align*}
where each $V_\ell$ is either one- or two-dimensional, and $U$ acts as multiplication by either $+1$ or $-1$ on every one-dimensional subspace and as a planar rotation on every two-dimensional subspace. By identifying each two-dimensional subspace $V_\ell$ with $\C$, there exists a unit complex number $\gamma_\ell \in\C$ so that $U$ acts as multiplication by $\gamma_\ell$ on that subspace. 
In other words, $U$ acts as multiplication by a unit $\gamma_\ell\in\C$ on every subspace $V_\ell$.

Let $\Pi_\ell$ denote projection onto $V_\ell$. Then,
\begin{align*}
    \Pi_\ell s_i = \frac{1+\mu}{2}\left(\Pi_\ell y_i - \Pi_\ell z_i\right)= \frac{1+\mu}{2}\left(\Pi_\ell y_i - \frac{1-\mu}{1+\mu}\Pi_\ell U y_i\right)= \left(\frac{1+\mu}{2} - \frac{1-\mu}{2} \gamma_\ell\right)\Pi_\ell y_i.
\end{align*}
Note that $\sigma_\ell\coloneqq \left(\frac{1+\mu}{2} - \frac{1-\mu}{2}\gamma_\ell\right)\in \Omega$.

The remainder of the proof is common to both settings.
Note that each quadratic form $Q_{M_i}(y_i, s_i)$ splits over orthogonal subspaces of $\R^d$, i.e., for $i= 0,N$,
\begin{align*}
    Q_{M_i}(y_i, s_i) &= \tr\left(\begin{bmatrix}
    y_i\\
    s_i
    \end{bmatrix}^\intercal M_i \begin{bmatrix}
    y_i\\
    s_i
    \end{bmatrix}\right)\\
    &= \sum_{\ell}\tr\left(\begin{bmatrix}
    \Pi_\ell y_i\\
    \Pi_\ell s_i
    \end{bmatrix}^\intercal M_i \begin{bmatrix}
    \Pi_\ell y_i\\
    \Pi_\ell s_i
    \end{bmatrix}\right)\\
    &=\sum_{\ell}\norm{\Pi_\ell y_i}^2 d_{\sigma_\ell}^* M_i d_{\sigma_\ell}.
\end{align*}
Finally, note that $\norm{\Pi_\ell y_N}^2 = \abs{\kappa(\sigma_\ell,\tilde H)}^2\norm{\Pi_\ell y_0}^2$ for all $\ell$.

We conclude that
\begin{align*}
    \frac{Q_{M_N}(y_N,s_N )}{Q_{M_0}(y_0,s_0 )} &\leq \max_{\sigma \in \Omega} \abs{\kappa(\sigma, \tilde H)}^2 \frac{d_\sigma^* M_N d_\sigma}{d_\sigma^* M_0 d_\sigma}.
\end{align*}

On the other hand, for any $\sigma\in \Omega$ and $y_0\neq 0$, the instance defined by $s_i = \sigma y_i$ is feasible. We conclude that \eqref{eq:tau_extremal_primal_dual} holds.
\end{proof}

\begin{proof}[Proof of \cref{lem:dd_gg_gd_certificate_form}]
We begin with the (D,D) statement.
Suppose $\Lambda$ is a PEP certificate and partition it as
\begin{align*}
    \Lambda &= \begin{bmatrix}
    \hat \Lambda & a\\
    b^\intercal & \lambda
    \end{bmatrix}
\end{align*}
where $\hat\Lambda\in\R^{N\times N}$ and $a,b\in\R^N$. By assumption $a, b \geq 0$ and 
\begin{align*}
    \hat\Lambda\bbm 1 + a  \leq 0,\quad \hat\Lambda^\intercal\bbm1 + b \leq 0,\quad 
    \bbm1^\intercal a + \lambda \leq 0,\quad
    \bbm1^\intercal b +\lambda\leq 0.
\end{align*}
In particular, $\lambda\leq 0$. If $\lambda = 0$, then $a = b = 0$ and we are done.
Otherwise, define
\begin{align*}
    \tilde\Lambda = \hat\Lambda - \frac{ab^\intercal}{\lambda} = \hat\Lambda + \frac{ab^\intercal}{\abs{\lambda}}.
\end{align*}
Note that $\hat\Lambda$ has nonnegative off-diagonal entries, $ab^\intercal$ is entrywise nonnegative and $\lambda <0$. Thus, $\tilde\Lambda$ has nonnegative off-diagonal entries.

The row and column sums of $\tilde\Lambda$ are 
\begin{align*}
    \tilde\Lambda\bbm1 &= \hat\Lambda\bbm1 + a \frac{\bbm1^\intercal b}{\abs{\lambda}} \leq \hat\Lambda\bbm1 + a  \leq 0\\
    \tilde\Lambda^\intercal\bbm1 &= \hat\Lambda^\intercal\bbm1 + b \frac{\bbm1^\intercal a}{\abs{\lambda}} \leq \hat\Lambda^\intercal\bbm1 + b  \leq 0.
\end{align*}

By assumption, $\Lambda$ is a PEP certificate, so
\begin{align*}
    \frac{\tau}{2}\norm{y_0}^2 - \frac{1}{2\tau}\norm{y_N}^2 + \tr(\bg^\intercal \Lambda \bx)\geq 0
\end{align*}
for all $(y_0,\bg)$. Let $\bar \bg$ and $\bar\bx$ denote the subvectors of $\bg$ and $\bx$ corresponding to entries $0,\dots,N-1$.
Then,
\begin{align*}
    0&\leq \frac{\tau}{2}\norm{y_0}^2 - \frac{1}{2\tau}\norm{y_N}^2 + \tr\left(\begin{bmatrix}
    \bar \bg\\
    g_N
    \end{bmatrix}^\intercal \begin{bmatrix}
    \hat\Lambda & a\\
    b^\intercal & \lambda
    \end{bmatrix} \begin{bmatrix}
    \bar \bx\\
    x_N
    \end{bmatrix}\right)\\
    &= \frac{\tau}{2}\norm{y_0}^2 - \frac{1}{2\tau}\norm{y_N}^2 + 
    \tr(\bar\bg^\intercal \hat\Lambda\bar\bx) + \ip{x_N, a^\intercal\bar\bg} + \ip{g_N, b^\intercal\bar \bx} + \lambda\ip{g_N,x_N}\\
    &=  \frac{\tau}{2}\norm{y_0}^2 - \frac{1}{2\tau}\norm{y_N}^2 + 
    \tr(\bar\bg^\intercal \hat\Lambda\bar\bx) + \frac{1}{1+q}\ip{y_N - g_N, a^\intercal\bar\bg} + \ip{g_N, b^\intercal\bar \bx} + \frac{\lambda}{1+q}\ip{g_N,y_N - g_N}.
\end{align*}
Here, the second line expands and the third line applies the identity $(1+q)x_N = y_N - g_N$.
Now, since $y_0, y_N, \bar \bg, \bar \bx$ are independent of $g_N$, we may set $g_N = -\frac{1}{\lambda}a^\intercal\bar\bg$. This gives
\begin{align*}
    \frac{\tau}{2}\norm{y_0}^2 - \frac{1}{2\tau}\norm{y_N}^2 + 
    \tr\left(\bar\bg^\intercal \left(\hat\Lambda - \frac{ab^\intercal}{\lambda}\right)\bar\bx\right)\geq0.
\end{align*}
Thus, $\textup{Slack}^\textup{D,D}_{\tau,\tilde\Lambda_{V}, \tilde H}$ is nonnegative for every $(y_0,\bar \bg)$. Finally, by inspection $\textup{Slack}^\textup{D,D}_{\tau,\tilde\Lambda_{V},\tilde H}$ is independent of $g_N$ so it is also nonnegative for every $(y_0,\bg)$.

An identical proof holds for the (G,D) setting upon replacing $\norm{y_0}^2$ with $\norm{s_0}^2$.

Now, consider the (G,G) setting.
Let $\Lambda$ be an arbitrary PEP certificate and set
\begin{align*}
    a = -\Lambda\bbm1,\quad b = -\Lambda^\intercal\bbm1,
\end{align*}
which are both nonnegative by assumption. Let $\rho = \bbm1^\intercal a = \bbm1^\intercal b = -\bbm1^\intercal\Lambda\bbm1$.
Thus $\rho \geq 0$.
If $\rho = 0$, then $a=b=0$ and we are done. Otherwise, define
\begin{align*}
    \Gamma = \Lambda + \frac{ab^\intercal}{\rho}.
\end{align*}

We check that $\Gamma\in \cP_\textup{cert}$. First, since $\Lambda$ has nonnegative off-diagonal entries, $a,b\geq 0$, and $\rho>0$, it holds that $\Gamma$ has nonnegative off-diagonal entries.
Next, the row and column sums of $\Gamma$ are
\begin{align*}
    \Gamma\bbm1 &= \Lambda\bbm1 + \frac{b^\intercal \bbm1}{\rho}a = -a + a = 0\\
\Gamma^\intercal\bbm1 &= \Lambda^\intercal\bbm1 + \frac{a^\intercal \bbm1}{\rho}b = -b + b = 0.
\end{align*}

Now, let $(y_0,\bs)$ be arbitrary and let $\bx, \by, \bg$ denote the associated data.
We will shift this trajectory's starting point while leaving $\bs$ unchanged. Set
\begin{align*}
    \delta = -\frac{b^\intercal \bx}{\rho},\quad
    y_0' = y_0 + \delta,\quad
    \bs' = \bs.
\end{align*}
Let $\bx',\by',\bg'$ denote the data associated with the starting point $y_0'$ and oracle responses $\bs'$.
It holds that
\begin{align*}
    \by' &= \by + \bbm 1\delta,\qquad
    \bx' = \bx + \bbm1 \delta,\qquad
    \bg' = \bg - q\bbm1 \delta.
\end{align*}

Now, as $\Lambda$ is a PEP certificate, it holds that
\begin{align*}
    0&\leq \textup{Slack}^\textup{G,G}_{\tau,\Lambda,\tilde H}(y_0', \bg')\\
    &= 
    \frac{\tau}{2}\norm{s_0'}^2 - \frac{1}{2\tau}\norm{s_N'}^2 + \tr((\bg')^\intercal\Lambda(\bx'))\\
    &= 
    \frac{\tau}{2}\norm{s_0}^2 - \frac{1}{2\tau}\norm{s_N}^2 + \tr((\bg - q\bbm1\delta)^\intercal\Lambda(\bx + \bbm 1\delta))\\
    &= \frac{\tau}{2}\norm{s_0}^2 - \frac{1}{2\tau}\norm{s_N}^2 + \tr(\bg^\intercal\Lambda\bx)
    -\ip{a^\intercal\bg,\delta} +q\ip{\delta, b^\intercal \bx} + q\rho\norm{\delta}^2 
    \\
    &= \frac{\tau}{2}\norm{s_0}^2 - \frac{1}{2\tau}\norm{s_N}^2 + \tr(\bg^\intercal\Lambda\bx)
    +\frac{1}{\rho}\ip{a^\intercal\bg,b^\intercal\bx}.
\end{align*}
This is exactly $\textup{Slack}^\textup{G,G}_{\tau,\Gamma,\tilde H}(y_0, \bg)$.
\end{proof}

\begin{proof}[Proof of \cref{lem:compressed_Lambda_P_Phi}]
Compute
\begin{align*}
        \Lambda_{U}P\Phi_{V} &= U\Lambda \left(U^\intercal P V\right) \Phi \left(V^\intercal P\right) P^{-1}\\
    &= U \Lambda \left(\begin{bmatrix}
    -\bbm1_N & I_N
    \end{bmatrix}\begin{bmatrix}
    0_N^\intercal & 1\\
    \bar P & \bbm 1_N
    \end{bmatrix}\begin{bmatrix}
    I_N\\
    0_N^\intercal
    \end{bmatrix}\right)\Phi \left(\begin{bmatrix}
    I_N & 0_N
    \end{bmatrix} \begin{bmatrix}
    0_N & \bar P\\
    1 & \bbm 1_N^\intercal
    \end{bmatrix}\right)P^{-1}\\
    &= U \Lambda \bar P (\bar P \Lambda \bar P)^{-1} \begin{bmatrix}
    0_N & \bar P
    \end{bmatrix}P^{-1}\\
    &= \begin{bmatrix}
    -\bbm1_N^\intercal\\
    I_N
    \end{bmatrix} \begin{bmatrix}
    0_N & I_N
    \end{bmatrix}P^{-1}\\
    &= (I - e_0\bbm 1^\intercal)P^{-1}.
\end{align*}

For the second identity, we begin by writing
\begin{align*}
    \Lambda_W \bar P_W \Phi_W &= W \Lambda \bar P\Phi W^\intercal= W\bar P^{-1} W^\intercal= \bar P^{-1}_W.
\end{align*}
Next, we note that
\begin{align*}
    \bar S^{-1} \bar R \bar S \bar R &= \bar S^{-1}\bar S^\intercal = \bar S^{-1} \left(\bbm1\bbm1^\intercal - \bar S + I\right)= e_0 \bbm1^\intercal - I + \bar S^{-1}.
\end{align*}

Thus,
\begin{align*}
    \bar P^{-1}_WP &=  \begin{bmatrix}
    \bar R\bar P^{-1} \bar R & \\
    & 0
    \end{bmatrix} \begin{bmatrix}
    0 & \bar P\\
    1 & \bbm1_N^\intercal
    \end{bmatrix}
    =  \begin{bmatrix}
    0 & \bar R \bar P^{-1}\bar R \bar P\\
    0 & 0
    \end{bmatrix}
    = \begin{bmatrix}
    0 & \bar S^{-1}\bar R \bar S \bar R\\
    0 & 0
    \end{bmatrix}\\
    &= \begin{bmatrix}
    0 & 1 & 1 & \dots & 1\\
    0 & -1 & 0 \\
    \vdots & & \ddots & \ddots\\
    0 & & & -1 & 0\\
    0 & & & & 0
    \end{bmatrix}\\
    &= e_0 \bbm1^\intercal - I + E_N.\qedhere
\end{align*}
\end{proof}

\section{Proofs of \cref{lem:DD_duality_blocks,lem:gd_blocks}}
\label{app:deferred_block_identities}

\begin{proof}[Proof of \cref{lem:DD_duality_blocks}]

The first identity is
\begin{align*}
    P^{-1}A^\vee P &=P^{-1}SH^A P = H^\intercal S^\intercal = A^\intercal - \frac{1}{q}e_0 v_H^\intercal.
\end{align*}

The second identity is
\begin{align*}
    P w^\vee &= P (I - (1+q)SH^A)^\intercal e_N = (I - (1+q)P^{-1}SH^A P)^\intercal \bbm1\\
    &=  (I - (1+q)SH) \bbm1\\
    &= \frac{1+q}{q}(I - q SH)\bbm1 - \frac{1}{q}\bbm1\\
    &= b - \frac{1}{q}\bbm1.
\end{align*}

The third identity is
\begin{align*}
    P^{-1}b^\vee &= P^{-1}v_{H^A} = P^{-1}(I - q SH^A)\bbm1 = P^{-1}(I - q SH^A)P e_N \\
    &= (I - q P^{-1}SH^AP)e_N= (I - q SH)^\intercal e_N = w + \eta e_0.
\end{align*}

The fourth identity is
\begin{align*}
    \zeta^\vee = (1+q)\eta_{H^A} = q \left(1 + \frac{1}{q}\right)\eta_H = \zeta.
\end{align*}

The remaining two identities follow from
\begin{align*}
    A\bbm 1 &= \left(SH + \frac{1}{q}v_H e_0^\intercal\right)\bbm1 = SH\bbm1 + \frac{1}{q}(I - q SH)\bbm1= \frac{1}{q}\bbm1\\
    \bbm1^\intercal w &= \bbm1^\intercal\left[(I - q SH)^\intercal e_N - \eta_H e_0\right] = v_H^\intercal e_N  - \eta_H = 0.\qedhere
\end{align*}
\end{proof}

\begin{proof}[Proof of \cref{lem:gd_blocks}]
We will use the following identities
\begin{align*}
    &\bar P_W = \bbm1\bbm1^\intercal - P && \bar P_W^{-1} = (e_0 \bbm1^\intercal - I + E_N)P^{-1}\\
    &\bar P_W^{-1}\bbm1 = e_0 && \bar P_W e_0 = \bbm1 - e_N\\
    & R(I - \bbm1 e_0^\intercal)P = S^\intercal(I-E_N)
\end{align*}
The second identity is \cref{lem:compressed_Lambda_P_Phi} and the remainder are direct calculations.

Now,
\begin{align*}
    SH^A + \frac{1}{q} v_{H^A}e_0^\intercal &= SH^A + \frac{1}{q}(I - qSH^A)\bbm1 e_0^\intercal = SH^A(I - \bbm 1e_0^\intercal) + \frac{1}{q}\bbm1 e_0^\intercal.
\end{align*}

Then,
\begin{align*}
    \bar P_W^{-1}A^\vee \bar P_W &= \bar P_W^{-1}\left(SH^A + \frac{1}{q}v_{H^A}e_0^\intercal\right) \bar P_W\\
    &= (e_0\bbm1^\intercal - I + E_N)P^{-1}\left(SH^A(I - \bbm 1e_0^\intercal) + \frac{1}{q}\bbm1 e_0^\intercal\right)(\bbm1\bbm1^\intercal - P)\\
    &= \left((e_0\bbm1^\intercal - I + E_N)H^\intercal S^\intercal(I-E_N)P^{-1} + \frac{1}{q}e_0 e_0^\intercal\right)(\bbm1\bbm1^\intercal - P)\\
    &= -(e_0\bbm1^\intercal - I + E_N)H^\intercal S^\intercal(I-E_N) + \frac{1}{q}e_0 \bbm1^\intercal (I - E_N)\\
    &= 
    \left[(I-E_N)H^\intercal S^\intercal + \frac{1}{q}e_0 \bbm 1^\intercal \left(I-qH^\intercal S^\intercal\right)\right] (I - E_N)\\
    &=     \left[H^\intercal S^\intercal + \frac{1}{q}e_0 v_H^\intercal \right] (I - E_N) = A^\intercal (I - E_N).
\end{align*}
Here, the second line substitutes for $\bar P_W^{-1}, \bar P_W$; the third line uses $R(I - \bbm1 e_0^\intercal) = S^\intercal (I - E_N)P^{-1}$; the fourth line uses $(I-E_N)P^{-1}\bbm1 = 0$; the final line uses the fact that $SH$ is lower triangular so that $(I- E_N)SHe_N = 0$.

The second identity is
\begin{align*}
    -\bar P_W w^\vee &= 
    -\bar P_W \left[(I - (1+q)SH^A)^\intercal e_N - \frac{1+q}{q}\eta e_0\right]\\
    &= 
    -\bar P_W \left[P^{-1}(I - (1+q)SH) \bbm1  - \frac{1+q}{q}\eta e_0\right]\\
    &= (I - \bbm1 e_N^\intercal) (I - (1+q)SH) \bbm1  + \frac{1+q}{q}\eta (\bbm1 - e_N)\\
    &= (I - \bbm1 e_N^\intercal) \left(\frac{1+q}{q}v_H - \frac{1}{q}\bbm1\right)  + \frac{1+q}{q}\eta (\bbm1 - e_N)\\
    &= \frac{1+q}{q}v_H  - \frac{1+q}{q}\eta e_N\\
    &= b - \frac{1+q}{q}\eta e_N.
\end{align*}
Here, the third line substitutes $\bar P_W$ and $\bar P_W e_0$ and the fifth line uses $\bbm 1\in\ker(I - \bbm1 e_N^\intercal)$.

The third identity is
\begin{align*}
    -\bar P_W^{-1}b^\vee &= - \frac{1+q}{q}\bar P_W^{-1}v_{H^A}
    = - \frac{1+q}{q}\bar P_W^{-1}(I - q SH^A)\bbm1\\
    &= - \frac{1+q}{q}\left(e_0 \bbm1^\intercal - I + E_N\right)P^{-1}(I - q SH^A)P e_N\\
    &= - \frac{1+q}{q}\left(e_0 \bbm1^\intercal - I + E_N\right)(I - q SH)^\intercal e_N\\
    &= - \frac{1+q}{q}\left(\eta e_0 -(I- q SH  - E_N + q SH E_N)^\intercal e_N\right)\\
    &= - \frac{1+q}{q}\left(\eta e_0 -\left(\frac{q}{1+q} I- q SH \right)^\intercal e_N\right)\\
    &= (I - (1+q)SH)^\intercal e_N- \frac{1+q}{q}\eta e_0 = w.
\end{align*}
Here, the second line substitutes $\bar P_W^{-1}$. The fifth lines uses that $SH$ is lower triangular with diagonal entries $\frac{1}{1+q}$, so that $SH E_N = \frac{1}{1+q}E_N$.

The fourth identity is
\begin{align*}
    \zeta^\vee &= (1+q)(1+\frac{1}{q})\eta_{H^A} = (1+q)(1+\frac{1}{q})\eta_{H} = \zeta.
\end{align*}

The remaining identity is
\begin{align*}
    e_N^\intercal w &= e_N^\intercal(I - (1+q)SH)^\intercal e_N = 0,
\end{align*}
where the last identity notes that $SH$ is lower triangular with diagonal entries $\frac{1}{1+q}$. The $e_N^\intercal w^\vee = 0$ proof is identical.
\end{proof}

\section{H-duality for the suboptimality criteria}
\label{app:suboptimality}
For completeness, this appendix extends our dualities between distance and gradient performance criteria to the suboptimality performance criteria. This is the setting in which H-duality was first observed or conjectured~\cite{Kim2024-Hduality,kim2016optimized,kim2021optimizing}.
Accordingly, we restrict our attention to the function setting in this appendix.

Let S denote the ``suboptimality-type'' criterion, so that (D,S) is an initial distance to terminal suboptimality setup and (S,G) is an initial suboptimality to terminal gradient setup.

\begin{definition}
Fix $q\geq 0$, $\tau>0$, and $\tilde H\in\R^{\cI_N\times \cI_N}$ strictly lower triangular. 
Let $\Lambda\in\R^{\cI_N\times \cI_N}$.
Set $\eta = e_N^\intercal v_H$ and define the following lifted slack matrices:
\begin{align*}
    \cL_{\Lambda,\tilde H}^\textup{D,S}&\coloneqq \begin{bmatrix}
    -2\sym(\Lambda SH) & \Lambda v_H & (I - q SH)^\intercal e_N & -qH^\intercal\bbm1 \\
    \cdot& \tau & q\eta & q\eta\\
    \cdot& \cdot& 1 & 0\\
    \cdot& \cdot& \cdot & q
    \end{bmatrix}\\
    \cL_{\Lambda,\tilde H}^\textup{S,G}&\coloneqq \begin{bmatrix}
    -2\sym(\Lambda SH) & \Lambda v_H & (I - q SH)^\intercal e_N & \Lambda\bbm1 \\
    \cdot& 1& q\eta & 1\\
    \cdot& \cdot& \tau &0 \\
    \cdot& \cdot& \cdot & 1+q
    \end{bmatrix},
\end{align*}
where the blocks below the diagonal are defined by symmetry.
We say that $\Lambda\in\R^{\cI_N\times \cI_N}$ is an unsigned certificate of (D,S) (respectively (S,G)) convergence with rate $\tau$ if $\cL_{\Lambda,\tilde H}^\textup{D,S}\succeq 0$ (respectively $\cL_{\Lambda,\tilde H}^\textup{S,G}\succeq 0$).
\end{definition}

\begin{lemma}
Suppose $\Lambda\in\R^{\cI_N\times \cI_N}$ is an unsigned certificate of (D,S) convergence of $\tilde H$ with rate $\tau$. If, in addition, $\Lambda$ has nonnegative off-diagonal entries, $\Lambda\bbm1 \leq 0$, and $\Lambda^\intercal\bbm 1 \leq -e_N$, then 
\begin{align*}
    h_\star - h_N + \frac{\tau}{2}\norm{y_0}^2 \geq 0
\end{align*}
on any trajectory generated by $\tilde H$.
\end{lemma}
\begin{proof}
Let $\eta = e_N^\intercal v_H$.
Applying 
\cref{lem:aggregated_interpolation} with $a = -\Lambda\bbm1$ and $b= -\Lambda^\intercal\bbm1 - e_N$ gives
\begin{align*}
    f_\star - f_N - \tr(\bg^\intercal \Lambda \bx) \geq 0.
\end{align*}

Note that 
    $s_N = e_N^\intercal(I-qSH)\bg + q\eta y_0$ 
    and
    $x_N = -\bbm1^\intercal H \bg + \eta y_0$.
Thus, letting $\cQ$ denote the Schur complement\footnote{In the case $q = 0$, one should first drop the row and column of the matrix before taking the Schur complement.} of $\cL_{\Lambda,\tilde H}^\textup{D,S}$ with respect to the bottom-right $2\times 2$ block gives
\begin{align*}
    \frac{1}{2}\tr\left(\begin{bmatrix}
    \bg\\
    y_0
    \end{bmatrix}^\intercal\cQ \begin{bmatrix}
    \bg\\
    y_0
    \end{bmatrix}\right) = 
    \tr(\bg^\intercal \Lambda\bx) + 
    \frac{\tau}{2}\norm{y_0}^2 - \frac{1}{2}\norm{s_N}^2 - \frac{q}{2}\norm{x_N}^2.
\end{align*}
This is nonnegative by the assumption that $\cL_{\Lambda,\tilde H}^\textup{D,S}\succeq 0$.

Summing the nonnegative displays completes the proof.
\end{proof}

\begin{lemma}
Suppose $\Lambda\in\R^{\cI_N\times \cI_N}$ 
is an unsigned certificate of (S,G) convergence of $\tilde H$ with rate $\tau$. If, in addition, $\Lambda$ has nonnegative off-diagonal entries, $\Lambda\bbm1 \leq 0$, $(\Lambda^\intercal\bbm 1)_i\leq 0$ for $i=1,\dots,N$, 
and $\sum_{i=1}^N (-\Lambda^\intercal\bbm1)_i \leq 1$,
then
\begin{align*}
     h_0 - h_\star  - \frac{1}{2\tau}\norm{s_N}^2 - \frac{1}{2(1+q)}\norm{\bbm1^\intercal\Lambda^\intercal\bg +g_0}^2\geq 0
\end{align*}
on any trajectory generated by $\tilde H$.
\end{lemma}
\begin{proof}
Let $\eta = e_N^\intercal v_H$, $\alpha = -\Lambda\bbm1$ and $\beta = -\Lambda^\intercal\bbm1$. By assumption, $\alpha\geq 0$, $\beta_1,\dots,\beta_N\geq 0$ and $\sum_{i=1}^{N}\beta_i\leq 1$.

Set $\tilde \Lambda = \Lambda + \alpha e_0^\intercal$. Note that $\tilde\Lambda$ has nonnegative off-diagonal entries, $\tilde\Lambda\bbm 1 = 0$ and $\tilde\Lambda^\intercal\bbm 1 = \Lambda^\intercal\bbm1 - (\bbm1^\intercal\Lambda\bbm1)e_0$.
Set
\begin{align*}
    b \coloneqq e_0 - \tilde\Lambda^\intercal\bbm1 = (1- \sum_{i=0}^N \beta_i) e_0 + \beta = \begin{bmatrix}
    1- \sum_{i=0}^N\beta_i + \beta_0\\
    \beta_1\\
    \vdots\\
    \beta_N
    \end{bmatrix}\geq 0
\end{align*}

Apply \cref{lem:aggregated_interpolation} with $\tilde\Lambda$, $a=0$ and $b$ to get
\begin{align*}
    f_0 - f_\star - \tr(\bg^\intercal(\Lambda - \Lambda\bbm1e_0^\intercal)\bx)\geq 0.
\end{align*}

Note that 
    $s_N = e_N^\intercal(I-qSH)\bg + q\eta y_0$.
Thus, the Schur complement $\cQ$ of $\cL_{\Lambda,\tilde H}^\textup{S,G}$ with respect to the bottom-right $2\times 2$ block satisfies
\begin{align*}
    \frac{1}{2}\tr\left(\begin{bmatrix}
    \bg\\
    y_0
    \end{bmatrix}^\intercal\cQ \begin{bmatrix}
    \bg\\
    y_0
    \end{bmatrix}\right) = 
    \tr(\bg^\intercal \Lambda\bx) + 
    \frac{1}{2}\norm{y_0}^2 - \frac{1}{2\tau}\norm{s_N}^2 - \frac{1}{2(1+q)}\norm{\bbm1^\intercal\Lambda^\intercal\bg +y_0}^2.
\end{align*}
This is nonnegative by the assumption that $\cL_{\Lambda,\tilde H}^\textup{S,G}\succeq 0$.

Summing the nonnegative displays gives
\begin{align*}
   f_0 - f_\star +\ip{\bbm1^\intercal\Lambda^\intercal\bg,x_0} + 
    \frac{1}{2}\norm{y_0}^2 - \frac{1}{2\tau}\norm{s_N}^2 - \frac{1}{2(1+q)}\norm{\bbm1^\intercal\Lambda^\intercal\bg +y_0}^2\geq 0.
\end{align*}
Now, using $y_0 = s_0 + x_0 = g_0 + (1+q)x_0$, we can rewrite the LHS as
\begin{align*}
   &f_0 - f_\star + 
    \frac{1}{2}\norm{s_0}^2
    + 
    \frac{q}{2}\norm{x_0}^2
    - \frac{1}{2\tau}\norm{s_N}^2 - \frac{1}{2(1+q)}\norm{\bbm1^\intercal\Lambda^\intercal\bg +g_0}^2\\
    &\qquad = h_0 - h_\star  - \frac{1}{2\tau}\norm{s_N}^2 - \frac{1}{2(1+q)}\norm{\bbm1^\intercal\Lambda^\intercal\bg +g_0}^2.\qedhere
\end{align*}
\end{proof}

\begin{theorem}
    \label{thm:DS}
    Suppose $q\geq 0$.
Let $\tilde H\in\R^{\cI_N\times \cI_N}$ be strictly lower triangular and let $\tau>0$. Let $\Lambda\in\R^{\cI_N\times \cI_N}$ be invertible and define
\begin{align*}
    \Phi \coloneqq (P\Lambda P)^{-1}\qquad\text{and}\qquad
    \cR \coloneqq \begin{bmatrix}
    \Lambda P\\
    & 0 & 1 & 0\\
     & 1 & 0 & 0\\
     & 1 & 0 & -1
    \end{bmatrix}
\end{align*}
Then,
\begin{align*}
    \cL_{\Lambda,\tilde H}^\textup{D,S} = \cR \cL_{\Phi,\tilde H^A}^\textup{S,G}\cR^\intercal.
\end{align*}
\end{theorem}
\begin{proof}

Throughout this proof, abbreviate $\cL_1 = \cL_{\Lambda,\tilde H}^\textup{D,S}$ and $\cL_2 = \cL_{\Phi,\tilde H^A}^\textup{S,G}$ and set $\eta = e_N^\intercal v_H = e_N^\intercal v_{H^A}$.
Below, we verify that $\cL_1 = \cR \cL_2\cR^\intercal$.

The (1,1), (1,2), and (1,3) blocks of $\cR \cL_2\cR^\intercal$ are exactly the (1,1), (1,2), and (1,3) blocks of $\cR \cL_2\cR^\intercal$ in the (D,G)--(D,G) setting (see \cref{thm:DG}). Thus, the proof of \cref{thm:DG} gives
\begin{align*}
    -2\Lambda P \sym(\Phi S H^A)P\Lambda^\intercal &= -2\sym(H^\intercal S^\intercal\Lambda^\intercal),\\
    \Lambda P (I - qSH^A)^\intercal e_N &= \Lambda v_H,\\
\Lambda P \Phi v_{H^A} &= (I - q SH)^\intercal e_N.
\end{align*}

The (1,4) block of $\cR\cL_2\cR^\intercal$ is
\begin{align*}
    \Lambda P \Phi(v_{H^A} - \bbm1) &= P^{-1}\left(I - q SH^A - I\right)\bbm1\\
    &= -q P^{-1} SH^A \bbm1\\
    &= -q P^{-1} PH^\intercal R\bbm1\\
    &= -q H^\intercal\bbm1.
\end{align*}

Finally, the bottom-right $3\times 3$ block is
\begin{align*}
    \begin{bmatrix}
    0 & 1 & 0\\
    1 & 0 & 0\\
    1 & 0 & -1
    \end{bmatrix}
    \begin{bmatrix}
    1& q\eta & 1\\
    q\eta& \tau & \\
    1 & & 1+q
    \end{bmatrix}\begin{bmatrix}
    0 & 1 & 1\\
    1 & 0 & 0\\
    0 & 0 & -1
    \end{bmatrix} &= \begin{bmatrix}
    \tau & q\eta &q\eta\\
        q\eta & 1 & 0\\
    q\eta & 0 & q
    \end{bmatrix}.\qedhere
\end{align*}
\end{proof}

\end{document}